%% file: poro_elaticity_v4.tex
\documentclass[final,leqno]{siamltex}
\usepackage{amsmath}
\usepackage{epsfig}
\usepackage{graphicx}
\usepackage{amssymb}

\usepackage{color}

\newtheorem{remark}{Remark}[section]

\def\lam{{\lambda}}

\def\Ome{{\Omega}}

\def\Del{{\Delta}}
\def\nab{{\nabla}}
\def\vepsi{{\varepsilon}}
\def\p{{\partial}}
\def\reff#1{\eqref{#1}}
\def\norm#1#2{\Vert\,#1\,\Vert_{#2}}

\def\vepsi{\varepsilon}

\def\cT{{\mathcal T}}

\def\no{{\nonumber}}

\def\div{{\mbox{\rm div\,}}}

\def\p{{\partial}}

\def\nab{\nabla}
\def\Ome{\Omega}
\def\lam{\lambda}
\def\Del{\Delta}

\newcommand{\bRM}{\mathbf{RM}}
\newcommand{\br}{\mathbf{r}}

\def\ba{\mathbf{a}}
\def\bb{\mathbf{b}}

\def\bE{\mathbf{E}}

\def\bC{\mathbf{C}}
\def\bbf{\mathbf{f}}
\def\bu{\mathbf{u}}

\def\bv{\mathbf{v}}
\def\bw{\mathbf{w}}

\def\bg{\mathbf{g}}
\def\bn{\mathbf{n}}
\def\bH{\mathbf{H}}
\def\bV{\mathbf{V}}
\def\bL{\mathbf{L}}
\def\bP{\mathbf{P}}

\def\bV{\mathbf{V}}

\def\bX{\mathbf{X}}

\def\R{\mathbb{R}}

\def\bx{{\bf x}}
\begin{document}


\title{Multiphysics Finite Element Methods for a Poroelasticity Model\footnote{Last
update: \today}}

\author{
Xiaobing Feng\thanks{
Department of Mathematics, The University of
Tennessee, Knoxville, TN 37996, U.S.A. ({\tt xfeng@math.utk.edu}).
The work of this author was partially supported by the NSF grants DMS-1016173 and DMS-1318486.}
\and 
Zhihao Ge\thanks{Institute of Applied Mathematics, School of Mathematics and Information 
Sciences, Henan University, Kaifeng, Henan Province 475004, P. R. China ({\tt zhihaoge@gmail.com}). 
The work of this author was partially supported by the National Natural Science Foundation of 
China grant \#10901047.}
\and
Yukun Li\thanks{Department of Mathematics, The University of
Tennessee, Knoxville, TN 37996, U.S.A. ({\tt yli@math.utk.edu}). The
work of this author was partially supported by the NSF grants DMS-1016173 and DMS-1318486.}
%
}

\maketitle


\setcounter{page}{1}



\begin{abstract}
This paper concerns with finite element approximations of a quasi-static poroelasticity model 
in displacement-pressure formulation which describes the dynamics of poro-elastic materials 
under an applied mechanical force on the boundary. To better describe the multiphysics 
process of deformation and diffusion 
for poro-elastic materials, we first present a reformulation of the original model by introducing 
two pseudo-pressures, one of them is shown to satisfy a diffusion equation,   
we then propose a time-stepping algorithm which decouples (or couples) the 
reformulated PDE problem at each time step 
into two sub-problems, one of which is a generalized Stokes problem for the displacement vector field 
(of the solid network of the poro-elastic material) along with one pseudo-pressure field
and the other is a diffusion problem for the other pseudo-pressure field (of the solvent 
of the material). To make this multiphysics approach feasible numerically, two critical 
issue must be resolved: the first one is the uniqueness 
of the generalized Stokes problem and the other is to find a good boundary condition for the diffusion 
equation so that it also becomes uniquely solvable. To address the first issue, we discover 
certain conserved quantities for the PDE solution which provide ideal candidates for a needed 
boundary condition for the pseudo-pressure field. The solution to the second issue is to use 
the generalized Stokes problem to generate a boundary condition for the diffusion problem. 
A practical advantage of the time-stepping algorithm allows one to use any convergent Stokes 
solver (and its code) together with any convergent diffusion equation solver (and its code) 
to solve the poroelasticity model. In the paper, the Taylor-Hood mixed finite element method 
combined with the $P_1$-conforming finite element method is used as an example to demonstrate 
the viability of the proposed multiphysics approach. It is proved that the solutions of the fully 
discrete finite element methods fulfill a discrete energy law which mimics the differential 
energy law satisfied by the PDE solution and converges optimally in the energy norm. 
Moreover, it is showed that the proposed formulation also has a built-in mechanism to overcome 
so-called ``locking phenomenon" associated with the numerical approximations of the poroelasticity model. 
Numerical experiments are presented to show the performance of the proposed approach and methods  
and to demonstrate the absence of ``locking phenomenon" in our numerical experiments.
The paper also presents a detailed PDE analysis for the poroelasticity model, especially, 
it is proved that this model converges to the well-known Biot's consolidation model from soil
mechanics as the constrained specific storage coefficient tends to zero. As a result, the proposed
approach and methods are robust under such a limit process.
\end{abstract}



\pagestyle{myheadings} 
\thispagestyle{plain} 
\markboth{XIAOBING FENG, ZHIHAO GE AND YUKUN LI}{MULTIPHYSICS FINITE ELEMENT METHODS FOR 
A POROELASTICITY MODEL}


\section{Introduction}\label{sec-1}
A poroelastic material (or medium) is a fluid-solid interaction system at pore scale and 
poromechanic is a branch of continuum mechanics and acoustics that studies the behavior of 
fluid-saturated porous materials. If the solid is an elastic material, then the subject
of the study is known as poroelasticity. Moreover, the elastic material may be governed by linear or 
nonlinear constitutive law, which then leads respectively to linear and nonlinear poroelasticity.
Examples of poroelastic materials include soil, polymer gels, and medicine pills, just name a 
few. Poroelastic materials exhibit an important state of matter found in a wide variety
of mechanical, biomedical and chemical systems (cf. \cite{coussy04,de86,tf79,terzaghi43,yd04b} and the 
references therein). They also possess some fascinating properties, in particular, they display 
thixotropy which means that they become fluid when agitated, but resolidify when resting. 
In general, the behavior of a poroelastic material is described by a multiphysics fluid-solid
interaction process at pore scale. Unlike standard (macroscopic) fluid-solid interaction systems, 
some physical phenomena of the multiphysics process of the poroelastic material may not be explicitly 
revealed in its mathematical model, instead, they are hidden in the model. This is indeed 
the case for the poroelasticity model to be studied in this paper.

This paper considers a general quasi-static model of linear poroelasticity which is broad enough 
to contain the well-known Biot's consolidation model from soil mechanics (cf. \cite{murad,lynn07})
and the Doi's model for polymer gels (cf. \cite{fh10,yd04b}). The quasi-static feature is due to  
the assumption that the acceleration of the solid (described by the second order time derivative 
of the displacement vector field) is assumed to be negligible.  We refer the reader to \cite{coussy04,
pw07,terzaghi43} for a derivation of the model and to \cite{schowalter00} for its mathematical analysis. 
When the parameter $c_0$, called the constrained specific storage coefficient, vanishes in the 
model, it reduces into the above mentioned Boit's model and Doi's model arising from two distinct 
applications. Their mathematical analysis can be found in \cite{fh10} and their finite element
numerical approximations based on two very different approaches were carried out in
\cite{murad,fh10}, respectively.  In \cite{pw07,pw07b} the authors proposed and analyzed 
a semi-discrete and a fully discrete mixed finite element method which simultaneously 
approximate the pressure and its gradient along with the displacement vector field. Since the 
implicit Euler scheme is used for the time discretization, a combined linear system 
must be solved at each time step.  It is observed in the numerical tests that the proposed 
fully discrete mixed finite method may exhibit a ``locking phenomenon" in the sense that 
the computed pressure oscillates and its accuracy deteriorate when a rapidly changed 
initial pressure is given, as explained in the \cite{pw09} that such a ``locking phenomenon"
is caused by the difficulty of satisfying the nearly divergence-free condition of
$\bu$ for very small time $t>0$. 

The goal of this paper is to present a multiphysics approach for approximating the 
poroelasticity model. A key idea of this approach is to derive a multiphysics 
reformulation for the original model which clearly reveals the underlying multiple physics 
process (i.e., the deformation and diffusion) of the pore-scale fluid-solid 
interaction system. To the end, two pseudo-pressures are introduced, one of them is shown 
to satisfy a diffusion equation, while the displacement vector field along with the other 
pseudo-pressure variable is shown to satisfy a generalized Stokes system. It should be noted
that the original pressure is eliminated in the reformulation, thus, it is not 
approximated as a primary (unknown) variable, instead, it is computed as a linear 
combination of the two pseudo-pressures. Based on this multiphysics reformulation
we then propose a time-stepping algorithm which decouples (or couples) the reformulated 
PDE problem at each time step into two sub-problems, a generalized Stokes problem for 
the displacement vector field along with a pseudo-pressures and a diffusion problem
for another pseudo-pressure field. 
To make this multiphysics approach feasible numerically, two critical issue must be 
resolved: the first one is the uniqueness of the generalized Stokes problem and the 
other is to find a good boundary condition for the diffusion equation so that it also 
becomes uniquely solvable. To overcome these difficulties, we discover certain conserved 
quantities for the PDE solution which can be imposed as needed boundary conditions
for the subproblems. 
Moreover, we demonstrate that, regardless the choice of discretization methods, the 
proposed formulation has a built-in mechanism to overcome the ``locking phenomenon" 
associated with numerical approximations of the poroelasticity model.
 
The remainder of this paper is organized as follows. In Section \ref{sec-2} we present a
complete PDE analysis of the poroelasticity model which emphasizes the energy law of 
the underlying model. Several conserved quantities are derived for the PDE solution.
Moreover, it is proved that the poroelasticity model converges to the Biot's consolidation
model as the constrained specific storage coefficient $c_0\to 0$.
In Section \ref{sec-3} we propose and analyze some fully discrete finite element methods 
based on the above mentioned multiphysics reformulation. Both coupled and 
decoupled time-stepping are considered and compared. The Taylor-Hood mixed finite element method
combined with the $P_1$-conforming finite element method is chosen as an example for spatial
discretization. It is proved that the solutions of the fully discrete finite element methods
fulfill a discrete energy law which mimics the differential energy law satisfied by the 
PDE solution. Optimal order error estimates
in the energy norm are also established. Finally, in Section \ref{sec-4}, several 
benchmark numerical experiments are provided to show the performance of the proposed 
approach and methods, and to demonstrate the absence of ``locking phenomenon" in 
our numerical experiments.

\section{PDE model and its analysis}\label{sec-2}

\subsection{Preliminaries} \label{sec-2.1}
$\Omega \subset \R^d \,(d=1,2,3)$ denotes a bounded polygonal domain with the boundary 
$\p\Ome$. The standard function space notation is adopted in this paper, their 
precise definitions can be found in \cite{bs08,cia,temam}.
In particular, $(\cdot,\cdot)$ and $\langle \cdot,\cdot\rangle$
denote respectively the standard $L^2(\Ome)$ and $L^2(\p\Ome)$ inner products.
For any Banach space $B$, we let $\mathbf{B}=[B]^d$,
and use $\mathbf{B}^\prime$ to denote its dual space. In particular,
we use $(\cdot,\cdot)_{\small\rm dual}$
and $\langle \cdot,\cdot \rangle_{\small\rm dual}$
to denote the dual product on $(\bH^1(\Ome))' \times \bH^1(\Ome)$,
and $\norm{\cdot}{L^p(B)}$ is a shorthand notation for
$\norm{\cdot}{L^p((0,T);B)}$.

We also introduce the function spaces
\begin{align*}
&L^2_0(\Omega):=\{q\in L^2(\Omega);\, (q,1)=0\}, \qquad \bX:= \bH^1(\Ome). 
\end{align*}
It is well known \cite{temam} that the following so-called
inf-sup condition holds in the space $\bX\times L^2_0(\Ome)$:
\begin{align}\label{e2.0}
\sup_{\bv\in \bX}\frac{(\div \bv,\varphi)}{\norm{\nab \bv}{L^2(\Ome)}}
\geq \alpha_0 \norm{\varphi}{L^2(\Ome)} \qquad \forall
\varphi\in L^2_0(\Ome),\quad \alpha_0>0.
\end{align}

Let 
\[
\bRM:=\{\br:=\ba+\bb \times x;\, \ba, \bb, x\in \R^d\}
\]
denote the space of infinitesimal rigid motions.
It is well known \cite{brenner,gra,temam} that $\bRM$ is the kernel of
the strain operator $\vepsi$, that is, $\br\in \bRM$ if and only if
$\vepsi(\br)=0$. Hence, we have
\begin{align}
\vepsi(\br)=0,\quad \div \br=0 \qquad\forall \br\in \bRM. \label{e4.100}
\end{align}

Let $\bL^2_\bot(\p\Ome)$ and $\bH^1_\bot(\Ome)$ denote respectively the 
subspaces of $\bL^2(\p\Ome)$ and $\bH^1(\Ome)$ which are orthogonal to $\bRM$, that is,
\begin{align*}
&\bH^1_\bot(\Ome):=\{\bv\in \bH^1(\Ome);\, (\bv,\br)=0\,\,\forall \br\in \bRM\},
\\
&\bL^2_\bot(\p\Ome):=\{\bg\in \bL^2(\p\Ome);\,\langle \bg,\br\rangle=0\,\,
\forall \br\in \bRM \}.
\end{align*}

It is well known \cite{dautray} that there exists a constant $c_1>0$ such that
\[
\inf_{\br\in \bRM}\|\bv+\br\|_{L^2(\Ome)}
\le c_1\|\vepsi(\bv)\|_{L^2(\Ome)} \qquad\forall \bv\in\bH^1(\Ome).
\]
Hence, for each $\bv\in \bH^1_\bot(\Ome)$ there holds
\begin{align}\label{e4.1+}
\|\bv\|_{L^2(\Ome)}=\inf_{\br\in \bRM} \sqrt{\|\bv+\br\|_{L^2(\Ome)}^2-\|\br\|_{L^2(\Ome)}^2 }
\le c_1\|\vepsi(\bv)\|_{L^2(\Ome)},
\end{align}
which and the well-known Korn's inequality \cite{dautray}
yield that for some $c_2>0$
\begin{align} \label{e4.1+0}
\|\bv\|_{H^1(\Ome)} &\le c_2[\|\bv\|_{L^2(\Ome)}+\|\vepsi(\bv)\|_{L^2(\Ome)}]\\
&\le c_2(1+c_1)\|\vepsi(\bv)\|_{L^2(\Ome)} \qquad\forall \bv\in\bH^1_\bot(\Ome).\no
\end{align}

By Lemma 2.1 of \cite{brenner} we know that for any $q\in L^2(\Ome)$, there
exists $\bv\in \bH^1_\bot(\Ome)$ such that $\div \bv=q$ and
$\|\bv\|_{H^1(\Ome)} \leq C\|q\|_{L^2(\Ome)}$. An immediate consequence of
this lemma is that there holds the following alternative version of the
inf-sup condition:
\begin{align}\label{e2.0a}
\sup_{\bv\in \bH^1_\bot(\Ome)}\frac{(\div \bv,\varphi)}{\norm{\nab \bv}{L^2(\Ome)}}
\geq \alpha_1 \norm{\varphi}{L^2(\Ome)} \qquad \forall
\varphi\in L^2_0(\Ome),\quad \alpha_1>0.
\end{align}

Throughout the paper, we assume $\Omega \subset \R^d$ is a bounded
polygonal domain such that $\Delta: H^1_0(\Omega) \cap H^2(\Omega)
\rightarrow L^2(\Omega)$ is an isomorphism
(cf. \cite{dauge88,gra}). In addition, $C$ is used to
denote a generic positive (pure) constant which may be different 
in different places. 

\subsection{PDE model and its  multiphysics reformulation} \label{sec-2.2}
The quasi-static poroelasticity model to be studied in this paper is given by (cf. \cite{pw07})
\begin{alignat}{2} \label{e1.1}
-\div \sigma(\bu) + \alpha \nab p &= \bbf
&&\qquad \mbox{in } \Ome_T:=\Ome\times (0,T)\subset \mathbf{\R}^d\times (0,T),\\
(c_0p+\alpha \div \bu)_t + \div \bv_f &=\phi &&\qquad \mbox{in } \Ome_T,
\label{e1.2}
\end{alignat}
where
\begin{align} \label{e1.3}
\sigma(\bu) &:= \mu\vepsi(\bu)+\lam \div\bu I,\qquad
\vepsi(\bu):=\frac12\bigl(\nab\bu+\nab\bu^T\bigr),\\
\bv_f &:= -\frac{K}{\mu_f} \bigl(\nab p -\rho_f \bg \bigr). \label{e1.4}
\end{align}
Where $\bu$ denotes the displacement vector of the solid and
$p$ denotes the pressure of the solvent. $\bbf$ is the body force.
$I$ denotes the $d\times d$ identity matrix and $\vepsi(\bu)$ is
known as the strain tensor. The parameters in the model are
Lam\'e constants $\lam$ and $\mu$, the (symmetric) permeability
tensor $K$, the solvent viscosity $\mu_f$, Biot-Willis constant
$\alpha$, and the constrained specific storage coefficient
$c_0$.  In addition, $\sigma(\bu)$ is called the (effective) stress tensor.
$\widehat{\sigma}(\bu,p):=\sigma(\bu)-\alpha p I$ is the total stress
tensor.  $\bv_f$ is the volumetric solvent flux and \reff{e1.4} is the
well-known Darcy's law. We assume that $\rho_f\not\equiv 0$, which 
is a realistic assumption.

To close the above system, suitable boundary and initial conditions must
also be prescribed. The following set of boundary and initial
conditions will be considered in this paper:
\begin{alignat}{2} \label{e1.4a}
\widehat{\sigma}(\bu,p)\bn=\sigma(\bu)\bn-\alpha p \bn &= \bbf_1
&&\qquad \mbox{on } \p\Ome_T:=\p\Ome\times (0,T),\\
\bv_f\cdot\bn= -\frac{K}{\mu_f} \bigl(\nab p -\rho_f \bg \bigr)\cdot \bn
&=\phi_1 &&\qquad \mbox{on } \p\Ome_T, \label{e1.4b} \\
\bu=\bu_0,\qquad p&=p_0 &&\qquad \mbox{in } \Ome\times\{t=0\}. \label{e1.4c}
\end{alignat}

We note that in some engineering literature the second Lam\'e constant
$\mu$ is also called the {\em shear modulus} and denoted by $G$, and
$B:=\lam +\frac23 G$ is called the {\em bulk modulus}. $\lam, \mu$ and $B$
are computed from the {\em Young's modulus} $E$ and the {\em Poisson ratio}
$\nu$ by the following formulas:
\[
\lam=\frac{E\nu}{(1+\nu)(1-2\nu)},\qquad \mu=G=\frac{E}{2(1+\nu)}, \qquad
B=\frac{E}{3(1-2\nu)}.
\]

Unlike the existing approaches in the literature \cite{pw07,murad}, 
in this paper we will not approximate the above original model directly,
instead, we first derive a (multiphysics) reformulation for the model, we then 
approximate the reformulated model. This is a key idea of this paper and
it will be seen in the later sections that this new approach is advantageous.  
To the end, we introduce new variables
\[
q:=\div \bu,\qquad \eta:=c_0p+\alpha q,\qquad \xi:=\alpha p -\lam q.
\]
It is easy to check that
\begin{align}\label{e1.5}
p=\kappa_1 \xi + \kappa_2 \eta, \qquad q=\kappa_1 \eta-\kappa_3 \xi,
\end{align}
where
\begin{align}\label{e1.6}
\kappa_1:= \frac{\alpha}{\alpha^2+\lam c_0},
\quad \kappa_2:=\frac{\lam}{\alpha^2+\lam c_0}, \quad
\kappa_3:=\frac{c_0}{\alpha^2+\lam c_0}.
\end{align}

Then \reff{e1.1}--\reff{e1.4} can be written as
\begin{alignat}{2} \label{e1.7}
-\mu\div\vepsi(\bu) + \nab \xi &= \bbf &&\qquad \mbox{in } \Ome_T,\\
\kappa_3\xi +\div \bu &=\kappa_1\eta &&\qquad \mbox{in } \Ome_T, \label{e1.8}\\
\eta_t - \frac{1}{\mu_f} \div[K (\nab (\kappa_1 \xi + \kappa_2 \eta)-\rho_f\bg)]&=\phi
&&\qquad \mbox{in } \Ome_T, \label{e1.9}
\end{alignat}
where $p$ and $q$ are related to $\xi$ and $\eta$
through the algebraic equations in \reff{e1.5}

It is now clear that $(\bu, \xi)$ satisfies a generalized Stokes problem
and $\eta$ satisfies a diffusion problem. This new formulation reveals 
the underlying deformation and diffusion multiphysics process which 
occurs in the poroelastic material. In particular, the diffusion part 
of the process is hidden in the original formulation but is apparent 
in the new formulation.  To make use the above reformulation for 
computation, we need to address a crucial issue of the uniqueness for
the generalized Stokes problem and the diffusion problem after they are
decoupled. The difficulty will be overcome by discovering some 
invariant quantities for the solution of the PDE model and using 
them to impose some appropriate boundary conditions for both subproblems
(cf. Lemma \ref{lem2.4}).

\subsection{Analysis of the PDE model}\label{sec-2.3}

We start this section with a definition of weak solutions to problem \reff{e1.1}--\reff{e1.4c}.
For convenience, we assume that $\bbf, \bbf_1,\phi$
and $\phi_1$ all are independent of $t$ in the remaining of the paper. We note that
all the results of this paper can be easily extended to the time-dependent case.

\begin{definition}\label{weak1}
Let $\bu_0\in\bH^1(\Ome), \bbf\in\bL^2(\Omega),
\bbf_1\in \bL^2(\p\Ome), p_0\in L^2(\Ome), \phi\in L^2(\Ome)$,
and $\phi_1\in L^2(\p\Ome)$.  Assume
$(\bbf,\bv)+\langle \bbf_1, \bv \rangle =0$ for any $\bv\in \mathbf{RM}$.
Given $T > 0$, a tuple $(\bu,p)$ with
\begin{alignat*}{2}
&\bu\in L^\infty\bigl(0,T; \bH_\perp^1(\Ome)),
&&\qquad p\in L^2 \bigl(0,T; H^1(\Omega)\bigr), \\
&(c_0p+\alpha \div\bu)_t \in L^2(0,T;H^{-1}(\Ome)),
&&\qquad c_0^{\frac12} p\in L^\infty\bigl(0,T; L^2(\Ome)),
\end{alignat*}
is called a weak solution to \reff{e1.1}--\reff{e1.4c},
if there hold for almost every $t \in [0,T]$
\begin{alignat}{2}\label{e2.1}
&\mu \bigl( \vepsi(\bu), \vepsi(\bv) \bigr)
+\lam\bigl(\div\bu, \div\bv \bigr)
-\alpha \bigl( p, \div \bv \bigr)  && \\
&\hskip 2in
=(\bbf, \bv)+\langle \bbf_1,\bv\rangle
&&\quad\forall \bv\in \bH^1(\Ome), \no \\
&\bigl((c_0 p +\alpha\div\bu)_t, \varphi \bigr)_{\rm dual}
+ \frac{1}{\mu_f} \bigl( K(\nab p-\rho_f\bg), \nab \varphi \bigr)
\label{e2.2} \\
&\hskip 2in =\bigl(\phi,\varphi\bigr)
+\langle \phi_1,\varphi \rangle
&&\quad\forall \varphi \in H^1(\Ome), \no  \\
&\bu(0) = \bu_0,\qquad p(0)=p_0.  && \label{e2.3}
\end{alignat}
\end{definition}

\medskip
Similarly, we can weak solutions to problem \reff{e1.7}--\reff{e1.9},
\reff{e1.4a}--\reff{e1.4c}.

\begin{definition}\label{weak2}
Let $\bu_0\in \bH^1(\Ome), \bbf \in \bL^2(\Omega),
\bbf_1 \in \bL^2(\p\Ome), p_0\in L^2(\Ome), \phi\in L^2(\Ome)$,
and $\phi_1\in L^2(\p\Ome)$.  Assume
$(\bbf,\bv)+\langle \bbf_1, \bv \rangle =0$ for any $\bv\in \mathbf{RM}$.
Given $T > 0$, a $5$-tuple $(\bu,\xi,\eta,p,q)$ with
\begin{alignat*}{2}
&\bu\in L^\infty\bigl(0,T; \bH_\perp^1(\Ome)), &&\qquad
\xi\in L^2 \bigl(0,T; L^2(\Omega)\bigr), \\
&\eta\in L^\infty\bigl(0,T; L^2(\Omega)\bigr)
\cap H^1\bigl(0,T; H^{-1}(\Omega)\bigr),
&&\qquad q\in L^\infty(0,T;L^2(\Ome)), \\
&p\in L^2 \bigl(0,T; H^1(\Omega)\bigr),  &&
\end{alignat*}
is called a weak solution to \reff{e1.7}--\reff{e1.9}, \reff{e1.4a}--\reff{e1.4c}
if there hold for almost every $t \in [0,T]$
\begin{alignat}{2}\label{e2.4}
\mu \bigl(\vepsi(\bu), \vepsi(\bv) \bigr)-\bigl( \xi, \div \bv \bigr)
&= (\bbf, \bv)+\langle \bbf_1,\bv\rangle
&&\quad\forall \bv\in \bH^1(\Ome), \\
\kappa_3 \bigl( \xi, \varphi \bigr) +\bigl(\div\bu, \varphi \bigr)  
&= \kappa_1\bigl(\eta, \varphi \bigr) &&\quad\forall \varphi \in L^2(\Ome), \label{e2.5}  \\
\bigl(\eta_t, \psi \bigr)_{\rm dual}
+\frac{1}{\mu_f} \bigl(K(\nab (\kappa_1\xi +\kappa_2\eta) &-\rho_f\bg), \nab \psi \bigr) \label{e2.6} \\
&= (\phi, \psi)+\langle \phi_1,\psi\rangle &&\quad\forall \psi \in H^1(\Ome) , \no  \\
p:=\kappa_1\xi +\kappa_2\eta, \qquad
&q:=\kappa_1\eta-\kappa_3\xi, && \label{e2.7} \\
\bu(0) = \bu_0, \qquad p(0) &=p_0, && \label{e2.8} \\
q(0)=q_0:=\div \bu_0,\quad \quad \eta(0)= \eta_0
:&=c_0p_0+\alpha q_0.  && \label{e2.9}
\end{alignat}
\end{definition}

\begin{remark}\label{rem-2.1}
(a) 
After $\xi$ and $\eta$ are computed, $p$ and $q$ are simply updated by their
algebraic expressions in \reff{e2.7}.

(b) Equation \reff{e2.6} implicitly imposes the following boundary condition for $\eta$:
\begin{align}\label{e2.010}
\kappa_2K\frac{\partial\eta}{\partial n}=K\rho_fg\cdot n-\kappa_1K\frac{\partial\xi}{\partial n}.
\end{align}

(c) It should be pointed out that the only reason for introducing the space $\bH_\perp^1(\Ome)$
in the above two definitions is that the boundary condition \eqref{e1.4a} 
is a pure ``Neumann condition". If it is replaced by a pure Dirichlet 
condition or by a mixed Dirichlet-Neumann condition, there is no need
to introduce this space. This fact will be used in our numerical experiments in 
Section \ref{sec-4}. We also note that from the analysis point of view,  the pure Neumann 
condition case is the most difficult case. 
\end{remark}

\begin{lemma}\label{lem2.1}
Every weak solution $(\bu,p)$ of problem \eqref{e2.1}--\eqref{e2.3} satisfies
the following energy law:
\begin{align}\label{e2.10}
E(t) + \frac{1}{\mu_f} \int_0^t \bigl( K(\nab p-\rho_f\bg), \nab p\bigr)\, ds 
-\int_0^t \bigl(\phi, p\bigr)\, ds
&-\int_0^t \langle \phi_1, p \rangle\, ds =E(0) 
\end{align}
for all $t\in [0,T]$,  where
\begin{align}\label{e2.10a}
E(t):&= \frac12 \Bigl[ \mu \norm{\vepsi(\bu(t))}{L^2(\Ome)}^2
+\lam \norm{\div \bu(t)}{L^2(\Ome)}^2
+c_0\norm{p(t)}{L^2(\Ome)}^2 \\
&\qquad
-2\bigl(\bbf,\bu(t)\bigr) -2\langle \bbf_1, \bu(t) \rangle \Bigr]. \no
\end{align}
Moreover,
\begin{align}\label{e2.11}
\norm{(c_0p+\alpha \div \bu)_t}{L^2(0.T;H^{-1}(\Ome))} 
&\leq\frac{K}{\mu_f} \norm{\nab p-\rho_f \bg}{L^2(\Ome_T)}  \\
&\qquad
+ \|\phi\|_{L^2(\Ome_T)} + \|\phi_1\|_{L^2(\p\Ome_T)} < \infty. \no
\end{align}

\end{lemma}

\begin{proof}
We only consider the case $\bu_t\in \bL^2((0, T);\bL^2(\Omega))$, the general case can be 
converted into this case using the Steklov average technique (cf. \cite[Chapter 2]{LSU}).
Setting $\varphi=p$ in \reff{e2.2} and $\bv=\bu_t$ in \reff{e2.1} yields for a.e. $t\in [0, T]$
\begin{alignat}{2}\label{e2.12}
&\bigl((c_0 p +\alpha\div\bu)_t, p(t) \bigr)_{\small\rm dual}
+ \frac{1}{\mu_f} \bigl( K(\nab p-\rho_f\bg), \nab p \bigr)
=\bigl(\phi,p\bigr) +\langle \phi_1, p \rangle, &&\\
&\mu \bigl( \vepsi(\bu), \vepsi(\bu_t) \bigr) +\lam\bigl(\div\bu, \div\bu_t \bigr)
-\alpha \bigl( p, \div \bu_t \bigr) 
=(\bbf, \bu_t)+\langle \bbf_1,\bu_t\rangle. \label{e2.13} &&
\end{alignat}

Adding the above two equations and integrating the sum in $t$ over the interval $(0, s)$ for 
any $s\in(0, T]$ yield
\begin{align}\label{e2.14}
E(s) + \frac{1}{\mu_f} \int_0^s \bigl(K (\nab p-\rho_f\bg), \nab p\bigr)\, dt
-\int_0^s \bigl(\phi, p\bigr)\, dt
-\int_0^s \langle \phi_1, p \rangle\, dt =E(0),
\end{align}
where $E(\cdot)$ is given by \eqref{e2.10a}. Here we have used the fact that $\bbf$ and $\bbf_1$ 
are independent of $t$.  Hence, \reff{e2.10} holds.

\eqref{e2.11} follows immediately from \eqref{e2.10} and \eqref{e2.2}. The proof is complete.
\end{proof}

Likewise, weak solutions of \eqref{e2.4}--\eqref{e2.9} satisfy a similar energy law which
is a rewritten version of \eqref{e2.10} in the new variables.

\begin{lemma}\label{lem2.2}
Every weak solution $(\bu,\xi,\eta,p,q)$ of problem \eqref{e2.4}--\eqref{e2.9} satisfies
the following energy law:
\begin{align}\label{e2.15a}
J(t) + \frac{1}{\mu_f} \int_0^t \bigl(K(\nab p-\rho_f\bg), \nab p\bigr)\, ds 
-\int_0^t \bigl(\phi, p\bigr)\, ds
&-\int_0^t \langle \phi_1, p \rangle\, ds =J(0) 
\end{align}
for all $t\in [0,T]$,  where
\begin{align}\label{e2.15b}
J(t):&= \frac12 \Bigl[ \mu \norm{\vepsi(\bu(t))}{L^2(\Ome)}^2
+\kappa_2 \norm{\eta(t)}{L^2(\Ome)}^2 +\kappa_3 \norm{\xi(t)}{L^2(\Ome)}^2 \\
&\qquad
-2\bigl(\bbf,\bu(t)\bigr) -2\langle \bbf_1, \bu(t) \rangle \Bigr]. \no
\end{align}
Moreover,
\begin{align}\label{e2.15c}
\norm{\eta_t}{L^2(0.T;H^{-1}(\Ome))} 
&\leq\frac{K}{\mu_f} \norm{\nab p-\rho_f \bg}{L^2(\Ome_T)}  \\
&\qquad
+ \|\phi\|_{L^2(\Ome_T)} + \|\phi_1\|_{L^2(\p\Ome_T)} < \infty. \no
\end{align}

\end{lemma}

\begin{proof} 
Again, we only consider the case that $\bu_t\in L^2(0,T; L^2(\Ome))$. Setting $\bv=\bu_t$ in \eqref{e2.4},
differentiating \eqref{e2.5} with respect to $t$ followed by taking $\varphi=\xi$, and setting 
$\psi=p=\kappa_1\xi +\kappa_2\eta$ in \eqref{e2.6}; adding the resulting equations and integrating in 
$t$ yield the desired equality \eqref{e2.15a}. The inequality \eqref{e2.15c} follows immediately from 
\eqref{e2.6} and \eqref{e2.15a}.
\end{proof}

The above energy law immediately implies the following solution estimates.

\begin{lemma}\label{estimates}
There exists a positive constant
$ C_1=C_1\bigl(\|\bu_0\|_{H^1(\Ome)}, \|p_0\|_{L^2(\Ome)},$ 
$\|\bbf\|_{L^2(\Ome)},\|\bbf_1\|_{L^2(\p \Ome)},\|\phi\|_{L^2(\Ome)}, \|\phi_1\|_{L^2(\p\Ome)} \bigr)$ 
such that
\begin{align}\label{e2.15d}
&\sqrt{\mu} \|\varepsilon(\bu)\|_{L^\infty(0,T; L^2(\Ome))} 
+\sqrt{\kappa_2} \|\eta\|_{L^\infty(0,T;L^2(\Ome))} \\
&\qquad
+\sqrt{\kappa_3} \|\xi\|_{L^\infty(0,T;L^2(\Ome))} 
+\sqrt{\frac{K}{\mu_f}} \|\nab p \|_{L^2(0,T;L^2(\Ome))} \leq C_1. \no \\
&\|\bu\|_{L^\infty(0,T;L^2(\Ome))}\leq C_1, \qquad 
\|p\|_{L^\infty(0,T;L^2(\Ome))} \leq C_1 \Bigl( 1+ \sqrt{\frac{\kappa_3}{\kappa_1}} \Bigr). \label{e2.15e}
\end{align}
\end{lemma}
We note that \eqref{e2.15e} follows from \eqref{e2.15d}, \eqref{e4.1+} and the relation 
$p=\kappa_1\xi +\kappa_2\eta$.

Furthermore, by exploiting the linearity of the PDE system, we have the following
a priori estimates for the weak solution.

\begin{theorem}\label{smooth}
Suppose that $\bu_0$ and $p_0$ are sufficiently smooth, then 
there exists a positive constant $C_2=C_2\bigl(C_1,\|\nab p_0\|_{L^2(\Ome)} \bigr)$ 
and $C_3=C_3\bigl(C_1,C_2, \|\bu_0\|_{H^2(\Ome)},\|p_0\|_{H^2(\Ome)} \bigr)$ such that 
there hold the following estimates for the solution to problem 
\reff{e1.7}--\reff{e1.9},\reff{e1.4a}--\reff{e1.4c}:
\begin{align}\label{eq3.21}
&\sqrt{\mu} \|\varepsilon(\bu_t)\|_{L^2(0,T; L^2(\Ome))} 
+\sqrt{\kappa_2} \|\eta_t\|_{L^2(0,T;L^2(\Ome))} \\
&\qquad
+\sqrt{\kappa_3} \|\xi_t\|_{L^2(0,T;L^2(\Ome))} 
+\sqrt{\frac{K}{\mu_f}} \|\nab p \|_{L^\infty(0,T;L^2(\Ome))} \leq C_2. \no \\
&\sqrt{\mu} \|\varepsilon(\bu_t)\|_{L^\infty(0,T; L^2(\Ome))} 
+\sqrt{\kappa_2} \|\eta_t\|_{L^\infty(0,T;L^2(\Ome)) \label{eq3.22}} \\
&\qquad
+\sqrt{\kappa_3} \|\xi_t\|_{L^\infty(0,T;L^2(\Ome))} 
+\sqrt{\frac{K}{\mu_f}} \|\nab p_t \|_{L^2(0,T;L^2(\Ome))} \leq C_3. \no \\
&\|\eta_{tt}\|_{L^2(H^{-1}(\Ome))} \leq \sqrt{\frac{K}{\mu_f}}C_3. \label{eq3.23}
\end{align}
\end{theorem}

\begin{proof}
On noting that $\bbf, \bbf_1,\phi$ and $\phi_1$ all are assumed to be independent of $t$, 
differentiating \eqref{e2.4} and \eqref{e2.5} with respect to $t$, taking $\bv=\bu_t$
and $\varphi=\xi_t$ in \eqref{e2.4} and \eqref{e2.5} respectively, and adding the 
resulting equations yield
\begin{align}\label{eq3.23a}
\mu\|\varepsilon(\bu_t)\|_{L^2(\Ome)}^2 = \bigl(q_t,\xi_t \bigr)
=\kappa_1 \bigl(\eta_t,\xi_t\bigr) -\kappa_3\|\xi_t\|_{L^2(\Ome)}^2.
\end{align}
Setting $\psi=p_t=\kappa_1\xi_t + \kappa_2\eta_t$ in \eqref{e2.6} gives 
\begin{align}\label{eq3.23b}
\kappa_1 \bigl(\eta_t,\xi_t \bigr) + \kappa_2\|\eta_t\|_{L^2(\Ome)}^2
+\frac{K}{2\mu_f} \frac{d}{dt} \|\nab p-\rho_f\bg\|_{L^2(\Ome)}^2 
=\frac{d}{dt}\Bigl[ (\phi,p) +\langle \phi_1, p\rangle \Bigr].
\end{align}

Adding \eqref{eq3.23a} and \eqref{eq3.23b} and integrating in $t$ we get for $t\in[0,T]$
\begin{align*}
&\frac{K}{2\mu_f} \|\nab p(t)-\rho_f\bg\|_{L^2(\Ome)}^2
+\int_0^t \Bigl[ \mu\|\varepsilon(\bu_t)\|_{L^2(\Ome)}^2
+ \kappa_2\|\eta_t\|_{L^2(\Ome)}^2 + \kappa_3\|\xi_t\|_{L^2(\Ome)}^2\Bigr]\,ds \\
&\hskip 0.65in
=\frac{K}{2\mu_f} \|\nab p_0-\rho_f\bg\|_{L^2(\Ome)}^2
+ (\phi,p(t)-p_0) +\langle \phi_1, p(t)-p_0 \rangle, \no 
\end{align*}
which readily infers \eqref{eq3.21}.

To show \eqref{eq3.22}, first differentiating \eqref{e2.4} one time with respect to $t$ and setting
$\bv=\bu_{tt}$, differentiating \eqref{e2.5} twice with respect to $t$ and setting 
$\varphi=\xi_t$, and adding the resulting equations we get
\begin{align}\label{eq3.23d}
\frac{\mu}{2}\frac{d}{dt}\|\varepsilon(\bu_t)\|_{L^2(\Ome)}^2 = \bigl(q_{tt},\xi_t \bigr)
=\kappa_1\bigl(\eta_{tt},\xi_t\bigr) -\frac{\kappa_3}{2}\frac{d}{dt}\|\xi_t\|_{L^2(\Ome)}^2.
\end{align}

Second, differentiating \eqref{e2.6} with respect $t$ one time and taking 
$\psi=p_t=\kappa_1\xi_t + \kappa_2\eta_t$ yield
\begin{align}\label{eq3.23e}
\kappa_1 \bigl(\eta_{tt},\xi_t \bigr) + \frac{\kappa_2}{2} \frac{d}{dt} \|\eta_t\|_{L^2(\Ome)}^2
+\frac{K}{\mu_f} \|\nab p_t\|_{L^2(\Ome)}^2 =0.
\end{align}

Finally, adding the above two inequalities and integrating in $t$ give for $t\in [0,T]$
\begin{align}\label{eq3.23f}
&\mu \|\varepsilon(\bu_t(t))\|_{L^2(\Ome)}^2 +\kappa_2 \|\eta_t(t)\|_{L^2(\Ome)}^2
+\kappa_3 \|\xi_t(t)\|_{L^2(\Ome)}^2 + \frac{2K}{\mu_f} \int_0^t \|\nab p_t\|_{L^2(\Ome)}^2\,ds\\
&\hskip 0.6in
=\mu \|\varepsilon(\bu_t(0))\|_{L^2(\Ome)}^2 +\kappa_2 \|\eta_t(0)\|_{L^2(\Ome)}^2
+\kappa_3 \|\xi_t(0)\|_{L^2(\Ome)}^2.\no
\end{align}
Hence, \eqref{eq3.22} holds.  \eqref{eq3.23} follows immediately from the following inequality
\begin{align*}
\bigl( \eta_{tt}, \psi \bigr) = -\frac{1}{\mu_f} \bigl( K\nab p, \nab \psi\bigr)
\leq \frac{K}{\mu_f} \|\nab p\|_{L^2(\Ome)} \|\nab \psi\|_{L^2(\Ome)}\qquad \forall \psi\in H^1_0(\Ome),
\end{align*}
\eqref{eq3.22} and the definition of the $H^{-1}$-norm. The proof is complete.
\end{proof}

\begin{remark}
As expected, the above high order norm solution estimates require $p_0\in H^1(\Ome),
\bu_t(0)\in \bL^2(\Ome), \eta_t(0)\in L^2(\Ome)$ and $\xi_t(0)\in L^2(\Ome)$.  The values 
of $\bu_t(0), \eta_t(0)$ and $\xi_t(0)$ can be computed using the PDEs as follows.
It follows from \eqref{e1.9} that $\eta_t(0)$ satisfies
\begin{align*}
\eta_t(0)= \phi + \frac{1}{\mu_f} \div[K (\nab p_0-\rho_f\bg)].
\end{align*}
Hence $\eta_t(0)\in L^2(\Ome)$ provided that $p_0\in H^2(\Ome)$.  

To find $\bu_t(0)$ and $\xi_t(0)$, differentiating \eqref{e1.7} and \eqref{e1.8} with 
respect to $t$ and setting $t=0$ we get 
\begin{alignat*}{2} 
-\mu\div\vepsi(\bu_t(0)) + \nab \xi_t(0) &=0 &&\qquad \mbox{\rm in } \Ome,\\
\kappa_3\xi_t(0) +\div \bu_t(0) &=\kappa_1\eta_t(0)  &&\qquad \mbox{\rm in } \Ome.
\end{alignat*}
Hence, $\bu_t(0)$ and $\xi_t(0)$ can be determined by solving the above generalized Stokes 
problem. 
\end{remark}

The next lemma shows that weak solutions of problem \reff{e2.4}--\reff{e2.9} preserve 
some ``invariant" quantities, it turns out that these ``invariant" quantities play a vital 
role in the construction of our time-splitting scheme to be introduced in the next section.

\begin{lemma}\label{lem2.4}
Every weak solution $(\bu,\xi,\eta,p, q)$ to problem \reff{e2.4}--\reff{e2.9} satisfies 
the following relations:
\begin{align}\label{e1.020}
&C_\eta(t):=\bigl(\eta(\cdot, t),1\bigr)
=\bigl(\eta_0,1\bigr) + \bigl[ (\phi,1) + \langle \phi_1, 1\rangle \bigr] t,\quad t\geq 0.\\
&C_\xi(t):=\bigl( \xi(\cdot,t), 1\bigr)=\frac{1}{d+\mu\kappa_3} \bigl[ \mu\kappa_1 C_\eta(t)
- \bigl(\bbf, x\bigr)- \langle \bbf_1,x\rangle \bigr].\label{e1.022} \\
&C_q(t):=\bigl( q(\cdot,t), 1\bigr) 
=\kappa_1 C_\eta(t)-\kappa_3 C_\xi(t). \label{e1.023}\\
&C_p(t):=\bigl( p(\cdot,t), 1\bigr)
=\kappa_1 C_\xi(t)+\kappa_2 C_\eta(t). \label{e1.021}\\
&C_\bu(t) := \bigl\langle \bu(\cdot,t)\cdot \bn, 1 \bigr\rangle  =C_q(t). \label{e1.024}
\end{align}
\end{lemma}

\begin{proof}
We first notice that equation \reff{e1.020} follows immediately from taking $\varphi\equiv 1$ 
in \reff{e2.6}, which is a valid test function.

To prove \reff{e1.022}, taking $\bv=x$ in \reff{e2.4} and $\varphi=1$ in \eqref{e2.5}, which are
valid test functions, and using the identities $\nabla x=I, \div x=d$, and $\varepsilon(x)=I$, we get
\begin{align}\label{e2.16}
\mu \bigl( \div \bu, 1\bigr) -d\bigl( \xi, 1\bigr) &=\bigl( \bbf, x\bigr)+ \langle \bbf_1,x\rangle,\\
\bigl( \div \bu, 1\bigr) &=\kappa_1(\eta,1)-\kappa_3(\xi, 1). \label{e2.6a}
\end{align}
Substituting \eqref{e2.6a} into \eqref{e2.16} and using \eqref{e1.020} yield
\begin{align}\label{e2.16b}
C_\xi(t):=\bigl( \xi(\cdot,t), 1\bigr) =\frac{1}{d+\mu\kappa_3} \bigl[ \mu\kappa_1 C_\eta(t)
- \bigl(\bbf, x\bigr)- \langle \bbf_1,x\rangle \bigr].
\end{align}
Hence \eqref{e1.022} holds. \eqref{e1.023} follows immediately from \eqref{e2.6a}, 
\reff{e1.020} and \eqref{e1.022}.

Finally, since $p=\kappa_1\xi + \kappa_2 \eta$, \eqref{e1.021} then follows from \reff{e1.020}
and \eqref{e1.022}. \eqref{e1.024} is an immediate consequence of $q=\div \bu$ and the divergence 
theorem. The proof is complete.
\end{proof}

\begin{remark}
We note that $C_\eta, C_\xi, C_q$ and $C_p$ all are (known) linear functions of $t$, and they become 
(known) constants when $\phi\equiv 0$ and $\phi_1\equiv 0$.
\end{remark}

With the help of the above lemmas, we can show the solvability of problem \reff{e1.1}--\reff{e1.4c}.

\begin{theorem}\label{thm2.5}
Let $\bu_0\in\bH^1(\Ome), \bbf\in\bL^2(\Omega),
\bbf_1\in \bL^2(\p\Ome), p_0\in L^2(\Ome), \phi\in L^2(\Ome)$, and $\phi_1\in L^2(\p\Ome)$.  
Suppose $(\bbf,\bv)+\langle \bbf_1, \bv \rangle =0$ for any $\bv\in \mathbf{RM}$.
Then there exists a unique solution to problem \reff{e1.1}--\reff{e1.4c} in the sense 
of Definition \ref{weak1}, likewise, there exists a unique solution to problem 
\reff{e1.7}--\reff{e1.9},\reff{e1.4a}--\reff{e1.4c} in the sense of Definition \ref{weak2}.
\end{theorem}

\begin{proof}
We only outline the main steps of the proof and leave the details to the interested reader. 

First, since the PDE system is linear, the existence of weak solution can be proved by
the standard Galerkin method and compactness argument (cf. \cite{temam}). We note that 
the energy laws established in Lemmas \ref{lem2.1} and \ref{lem2.2} guarantee the required
uniform estimates for the Galerkin approximate solutions. 

Second, to show the uniqueness, suppose there are two sets of weak solutions, again by the 
linearity of the PDE system it is trivial to show that the difference of the solutions satisfy the 
same PDE system with {\em zero} initial and boundary data. The energy law immediately 
implies that the difference must be zero, hence, the uniqueness is verified. 
\end{proof}

We conclude this section by establishing a convergence result for the solution of 
problem \reff{e1.7}--\reff{e1.9},\reff{e1.4a}--\reff{e1.4c} when the constrained specific 
storage coefficient $c_0$ tends to $0$. Such a convergence result is useful and significant
for the following two reasons.  First, as mentioned earlier, the poroelasticity model
studied in this paper reduces into the well-known Biot's consolidation model from 
soil mechanics (cf. \cite{murad,pw07}) and Doi's model for polymer gels (cf. \cite{yd04b,fh10}).
Second, it proves that the proposed approach and methods of this paper are robust under such 
a limit process.
 
\begin{theorem}\label{thm2.6}
Let $\bu_0\in\bH^1(\Ome), \bbf\in\bL^2(\Omega),
\bbf_1\in \bL^2(\p\Ome), p_0\in L^2(\Ome), \phi\in L^2(\Ome)$, and $\phi_1\in L^2(\p\Ome)$.
Suppose $(\bbf,\bv)+\langle \bbf_1, \bv \rangle =0$ for any $\bv\in \mathbf{RM}$.
Let $(\bu_{c_0},\eta_{c_0},\xi_{c_0},p_{c_0},q_{c_0})$ denote the unique weak 
solution to problem \reff{e1.7}--\reff{e1.9},\reff{e1.4a}--\reff{e1.4c}. Then there exists
$(\bu_*, \eta_*,\xi_*, p_*, q_*)\in \bL^\infty(0,T;\bH^1_\perp(\Ome))\times L^\infty(0,T;L^2(\Ome))\times 
L^\infty(0,T; L^2(\Ome)))\times L^\infty(0,T;L^2(\Ome))\cap L^2(0,T;H^1(\Ome)\times L^\infty(0,T;L^2(\Ome))$ 
such that $(\bu_{c_0},\eta_{c_0},\xi_{c_0},p_{c_0},q_{c_0})$ converges weakly to
$(\bu_*, \eta_*,\xi_*, p_*, q_*)$ in the above product space as $c_0\to 0$.
\end{theorem}

\begin{proof}
It follows immediately from \eqref{e2.15c}--\eqref{e2.15e} and Korn's inequality that 
\begin{itemize}
\item $\bu_{c_0}$ is uniformly bounded (in $c_0$) in $\bL^\infty(0,T; \bH^1_\perp(\Ome))$.
\item $\sqrt{\kappa_2} \eta_{c_0}$ is uniformly bounded (in $c_0$) in 
$L^\infty(0,T; L^2(\Ome))\cap L^2(0,T; H^{-1}(\Ome))$.
\item $\sqrt{\kappa_3}\xi_{c_0}$ is  uniformly bounded (in $c_0$) in $L^\infty(0,T; L^2(\Ome))$.
\item $p_{c_0}$ is uniformly bounded (in $c_0$) in $L^\infty(0,T; L^2(\Ome)) \cap L^2(0,T; H^1(\Ome))$.
\item $q_{c_0}$ is uniformly bounded (in $c_0$) in $L^\infty(0,T; L^2(\Ome))$.
\end{itemize}

On noting that $\lim_{c_0\to 0}\kappa_1=\frac{1}{\alpha}$, 
$\lim_{c_0\to 0}\kappa_2=\frac{\lam}{\alpha^2}$ and $\lim_{c_0\to 0}\kappa_3=0$,
by the weak compactness of reflexive Banach spaces and Aubin-Lions Lemma \cite{dautray}
we have that there exist 
$(\bu_*, \eta_*,\xi_*, p_*, q_*)\in \bL^\infty(0,T;\bH^1_\perp(\Ome))
\times L^\infty(0,T;L^2(\Ome))\times L^\infty(0,T; L^2(\Ome)))\times 
L^\infty(0,T;$ $L^2(\Ome))\cap L^2(0,T;H^1(\Ome)\times L^\infty(0,T;L^2(\Ome))$ and
a subsequence of $(\bu_{c_0},\eta_{c_0},\xi_{c_0},$ $p_{c_0},q_{c_0})$ (still denoted by the 
same notation) such that as $c_0\to 0$ (a subsequence of $c_0$, to be exact)
\begin{itemize}
\item $\bu_{c_0}$ converges to $\bu_*$ weak $*$ in $\bL^\infty(0,T; \bH^1_\perp(\Ome))$ and 
weakly in $\bL^2(0,T; \bH^1_\perp(\Ome))$.
\item $\sqrt{\kappa_2} \eta_{c_0}$ converges to $\frac{\sqrt{\lam}}{\alpha} \eta_*$ weak $*$ in 
$L^\infty(0,T; L^2(\Ome))$ and strongly in $L^2(\Ome_T)$.
\item $\kappa_3\xi_{c_0}$ converges to $0$ strongly in  $L^2(\Ome_T)$. 
\item $p_{c_0}$ converges to $p_*$ weak $*$ in $L^\infty(0,T; L^2(\Ome))$ and weakly in 
$L^2(0,T; H^1(\Ome))$.
\item $q_{c_0}$ converges to $p_*$ weak $*$ in $L^\infty(0,T; L^2(\Ome))$ and weakly in 
$L^2(\Ome_T)$.
\end{itemize}
Then setting $c_0\to 0$ in \eqref{e2.4}--\eqref{e2.9} yields (note that the dependence of the
solution on $c_0$ is suppressed there)
\begin{alignat*}{2}
\mu \bigl(\vepsi(\bu_*), \vepsi(\bv) \bigr)-\bigl( \xi_*, \div \bv \bigr)
&= (\bbf, \bv)+\langle \bbf_1,\bv\rangle &&\qquad\forall \bv\in \bH^1(\Ome), \\
\bigl(\div\bu_*, \varphi \bigr) &= \frac{1}{\alpha}\bigl(\eta_*, \varphi \bigr)
&&\qquad\forall \varphi \in L^2(\Ome),  \\
\bigl((\eta_*)_t, \psi \bigr)_{\rm dual}
+\frac{1}{\mu_f} \bigl(K(\nab (\alpha^{-1} +\lambda\alpha^{-2}\eta_*) &-\rho_f\bg), \nab \psi \bigr) &&\\
&= (\phi, \psi)+\langle \phi_1,\psi\rangle &&\qquad\forall \psi \in H^1(\Ome) ,  \\
p_*:=\frac{1}{\alpha}\xi_* +\frac{\lam}{\alpha^2}\eta_*, \qquad
&q_*:=\frac{1}{\alpha} \eta_*, &&  \\
&\bu_*(0) = \bu_0, \qquad  &&  \\
q_*(0)=q_0:=\div \bu_0,\quad \quad \eta_*(0)= \eta_0 :&=\alpha q_0.  && 
\end{alignat*}
Equivalently,
\begin{alignat*}{2}
\mu \bigl(\vepsi(\bu_*), \vepsi(\bv) \bigr)-\bigl( \xi_*, \div \bv \bigr)
&= (\bbf, \bv)+\langle \bbf_1,\bv\rangle
&&\quad\forall \bv\in \bH^1(\Ome), \\
\bigl(\div\bu_*, \varphi \bigr) &= \bigl(q_*, \varphi \bigr)
&&\quad\forall \varphi \in L^2(\Ome),  \\
\alpha \bigl((q_*)_t, \psi \bigr)_{\rm dual}
+\frac{1}{\mu_f} \bigl(K(\nab p_*-\rho_f\bg), \nab \psi \bigr)
&= (\phi, \psi)+\langle \phi_1,\psi\rangle
&&\quad\forall \psi \in H^1(\Ome) ,  \\
p_*:=\frac{1}{\alpha} \Bigl(\xi_* +\lam q_* \Bigr) \quad\mbox{or}\quad 
&\xi_*=\alpha p_*-\lam q_*, \qquad  &&  \\
&\bu_*(0) = \bu_0, \qquad  &&  \\
&q_*(0)=q_0:=\div \bu_0.\qquad   && 
\end{alignat*}
Hence, $(\bu_*, \eta_*,\xi_*,p_*,q_*)$ is a weak solution of Biot's consolidation model 
(cf. \cite{yd04b,fh10}). By the uniqueness of its solutions,  we conclude that the whole sequence 
$(\bu_{c_0},\eta_{c_0},\xi_{c_0},$ $p_{c_0},q_{c_0})$ converges to $(\bu_*, \eta_*,\xi_*,p_*,q_*)$ 
as $c_0\to 0$ in the above sense. The proof is complete.
\end{proof}

\input section_3b.tex

\input section_4a.tex


\end{document}

%% file: section_3b.tex
\section{Fully discrete finite element methods}\label{sec-3}
The goal of this section is to design and analyze some fully discrete finite
element methods for the poroelasticity model based on the above new formulation.
As the time stepping is vital for the overall methods, we first introduce our 
time-stepping schemes at the PDE level.

\subsection{Basic time-stepping algorithm}\label{sec-3.1}

Based on this new formulation, our multiphysics time-stepping  algorithm
reads as follows:

\bigskip
{\bf Splitting Algorithm (SA):} \\

\begin{itemize}
\item[(i)]  Set
\begin{alignat*}{3}
\bu^0 &=\bu_0, &&\qquad p^0 =p_0, \quad q^0 =q_0:=\div \bu_0, \\
\eta^0 &=c_0p^0+\alpha q^0, &&\qquad \xi^0=\alpha p^0-\lambda q^0.
\end{alignat*}

\item[(ii)] For $n=0,1,2, \cdots$,  complete the following three steps:\\

{\em Step 1:} Solve for $(\bu^{n+1},\xi^{n+1})$ such that
\begin{alignat}{2} \label{e1.19}
-\mu \div\vepsi(\bu^{n+1}) + \nab \xi^{n+1} &=\bbf,\quad
&&\qquad\mbox{in } \Ome_T, \\
\kappa_3\xi^{n+1} +\div \bu^{n+1} &=\kappa_1 \eta^{n+\theta} &&\qquad\mbox{in } \Ome_T,  \label{e1.20} \\
\widetilde{\sigma}(\bu^{n+1},\xi^{n+1})\bn &=\bbf_1
&&\qquad\mbox{on } \p\Ome_T.  \label{e1.21}
\end{alignat}
Here $\theta=0$ or $1$.

\medskip
{\em Step 2:} Solve for $\eta^{n+1}$ such that
\begin{alignat}{2} \label{e1.23}
d_t \eta^{n+1}-\frac{1}{\mu_f}\div[K(\nab(\kappa_2\eta^{n+1}
+\kappa_1\xi^{n+1}) -\rho_f \bg)] &=\phi, &&\qquad\mbox{in } \Ome_T,   \\
\frac{1}{\mu_f} K [\nab(\kappa_2\eta^{n+1}
+\kappa_1\xi^{n+1}) -\rho_f\bg]\cdot \bn
&=\phi_1 &&\qquad\mbox{on } \p\Ome_T.  \label{e1.24}
\end{alignat}

{\em Step 3:} Update $p^{n+1}$ and $q^{n+1}$ by
\begin{align}\label{e1.26}
p^{n+1} = \kappa_1\xi^{n+1} + \kappa_2\eta^{n+\theta}, \qquad
q^{n+1} = \kappa_1\eta^{n+1}-\kappa_3\xi^{n+1}.
\end{align}
Where $d_t \eta^{n+1}:= (\eta^{n+1}-\eta^n)/{\Del t}$,
$\Del t$ denotes the time step size of a uniform partition
of the time interval $[0,T]$, and
\begin{align}\label{e1.27}
\widetilde{\sigma}(\bu^{n+1},\xi^{n+1}):=\mu\vepsi(\bu^{n+1})- \xi^{n+1}I.
\end{align}
\end{itemize}
We note that \eqref{e1.23} is the implicit Euler scheme, which
is chosen just for the ease of presentation, it can be replaced by
other time-stepping schemes. \eqref{e1.24} provides a flux
boundary condition for $\eta^{n+1}$.

\begin{remark}
When $\theta=0$, {\em Step 1} and {\em Step 2} are decoupled, hence these two sub-problems can be solved
independently. On the other hand, when $\theta=1$, two sub-problems are coupled, 
hence, they must be solved together. 
\end{remark}

\subsection{Fully discrete finite element methods}\label{sec-3.2}

In this section, we consider the space-time discretization which combines the above  
splitting algorithm with appropriately chosen spatial discretization methods. To 
the end, we introduce some notation.
Assume $\Omega\in\mathbb{R}^d (d=2, 3)$ is a polygonal domain. Let $\mathcal{T}_h$ be a 
quasi-uniform triangulation or rectangular partition of $\Omega$ with mesh size $h$, 
and $\bar{\Omega}=\bigcup_{K\in\mathcal{T}_h}\bar{K}$. Also, let $(\bX_h, M_h)$ be a stable 
mixed finite element pair, that is, $\bX_h\subset \bH^1(\Omega)$ and $M_h\subset L^2(\Omega)$ 
satisfy the inf-sup condition
\begin{alignat}{2}\label{1407-1}
\sup_{v_h\in \bX_h}\frac{({\rm div} v_h, \varphi_h)}{\|\nabla v_h\|_{L^2(\Ome)}}
\geq \beta_0\|\varphi_h\|_{L^2(\Ome)} &&\quad \forall\varphi_h\in M_{0h}:=M_h\cap L_0^2(\Omega),\ \beta_0>0.
\end{alignat}

A number of stable mixed finite element spaces $(\bX_h, M_h)$ have been known in the literature
\cite{brezzi}.  A well-known example is the following
so-called Taylor-Hood element (cf. \cite{ber,brezzi}):
\begin{align*}
\bX_h &=\bigl\{\bv_h\in \bC^0(\overline{\Ome});\,
\bv_h|_K\in \bP_2(K)~~\forall K\in \cT_h \bigr\}, \\
M_h &=\bigl\{\varphi_h\in C^0(\overline{\Ome});\, \varphi_h|_K\in P_1(K)
~~\forall K\in \cT_h \bigr\}.
\end{align*}
In the next subsection, we shall only present the analysis for the
Taylor-Hood element, but remark that the analysis can be 
extended to other stable mixed elements. However, piecewise constant
space $M_h$ is not recommended because that would result in no
rate of convergence for the approximation of the pressure $p$ (see
Section \ref{sec-3.4}).

Finite element approximation space $W_h$ for $\eta$ variable can be chosen
independently, any piecewise polynomial space is acceptable provided that
$W_h \supset M_h$. Especially, $W_h\subset L^2(\Ome)$ can be
chosen as a fully discontinuous piecewise polynomial space, although
it is more convenient to choose $W_h$ to be a continuous (resp.
discontinuous) space if $M_h$ is a continuous (resp. discontinuous)
space. The most convenient choice is $W_h =M_h$, which
will be adopted in the remainder of this paper.

Recall that $\bRM$ denotes the space of the infinitesimal rigid motions
(see Section \ref{sec-2}), evidently, $\bRM\subset \bX_h$. We now
introduce the $L^2$-projection $\mathcal{P}_R$ from $\bL^2(\Ome)$
to $\bRM$. For each $\bv\in \bL^2(\Ome)$, $\mathcal{P}_R\bv_h\in \bRM$ is defined by
\[
(\mathcal{P}_R\bv_h,\br)=(\bv_h,\br) \qquad\forall \br\in \bRM.
\]
Moreover, we define
\begin{equation}\label{e3.50}
\bV_h:=(I-\mathcal{P}_R)\bX_h =\bigl\{\bv_h\in \bX_h;\,  (\bv_h,\br)=0\,\,
\forall \br\in \bRM \bigr\}.
\end{equation}
It is easy to check that $\bX_h=\bV_h\bigoplus \bRM$. It was proved in \cite{fh10}
that there holds the following alternative version of the above inf-sup condition:
\begin{align}\label{e3.51}
\sup_{\bv_h\in \bV_h}\frac{(\div\bv_h,\varphi_h)}{\norm{\nab \bv_h}{L^2(\Ome)}} 
\geq \beta_1 \norm{\varphi_h}{L^2(\Ome)} \quad \forall \varphi_h\in M_{0h}, \quad \beta_1>0.
\end{align}

\bigskip
{\bf Finite Element Algorithm (FEA):} \\

\begin{itemize}
\item[(i)]
Compute $\bu^0_h\in \bV_h$ and $q^0_h\in W_h$ by
\begin{alignat*}{3}
\bu^0_h &=\mathcal{R}_h\bu_0, &&\qquad p^0_h =\mathcal{Q}_hp_0, \qquad
 q^0_h &=\mathcal{Q}_hq_0 \ (q_0 =\div \bu_0), \\
\eta^0_h &=c_0p^0_h+\alpha q^0_h, &&\qquad 
\xi_h^0 =\alpha p_h^0 -\lambda q_h^0.
\end{alignat*}

\item[(ii)] For $n=0,1,2, \cdots$,  do the following three steps.

{\em Step 1:} Solve for $(\bu^{n+1}_h,\xi^{n+1}_h)\in \bV_h\times W_h$ such that
\begin{alignat}{2}\label{e3.1}
&\mu \bigl(\vepsi(\bu^{n+1}_h), \vepsi(\bv_h) \bigr)-\bigl( \xi^{n+1}_h, \div \bv_h \bigr)
= (\bbf, \bv_h)+\langle \bbf_1,\bv_h\rangle &&\quad \forall \bv_h\in \bV_h, \\
&\kappa_3\bigl(\xi^{n+1}_h, \varphi_h \bigr) +\bigl(\div\bu^{n+1}_h, \varphi_h \bigr)
=\kappa_1\bigl( \eta^{n+\theta}_h, \varphi_h \bigr)  
&&\quad \forall \varphi_h \in M_h. \label{e3.2}
\end{alignat}
{\em Step 2:} Solve for $\eta^{n+1}_h\in W_h$ such that
\begin{alignat}{2}\label{e3.4}
\bigl(d_t\eta^{n+1}_h, \psi_h \bigr)
&+\frac{1}{\mu_f} \bigl(K(\nab (\kappa_1\xi^{n+1}_h 
+\kappa_2\eta^{n+1}_h)-\rho_f\bg,\nab\psi_h \bigr)\\
&=(\phi, \psi_h)+\langle \phi_1,\psi_h\rangle. \no
\end{alignat}
{\em Step 3:} Update $p^{n+1}_h$ and $q^{n+1}_h$ by
\begin{alignat}{2}
p^{n+1}_h&=\kappa_1\xi^{n+1}_h +\kappa_2\eta^{n+\theta}_h, \quad
q^{n+1}_h=\kappa_1\eta^{n+1}_h-\kappa_3\xi^{n+1}_h.  \label{e3.5}
\end{alignat}
\end{itemize}

\begin{remark}\label{rem3.1}
At each time step, problem \reff{e3.1}--\reff{e3.2} is a generalized Stokes problem with 
a mixed boundary condition for $(\bu,p)$. The well-posedness of the generalized Stokes 
problem follows easily with the help of the inf-sup condition.
%
%
\end{remark}

\subsection{Stability analysis of fully discrete finite element methods}\label{sec-3.3}

The primary goal of this subsection is to derive a discrete energy law which mimics 
the PDE energy law \reff{e2.10}. It turns out that such a discrete energy law only 
holds if $h$ and $\Delta t$ satisfy the mesh constraint $\Delta t=O(h^2)$ when
$\theta=0$ but for all $h,\Delta t>0$ when $\theta=1$.

Before discussing the stability of (FEA), We first show that the numerical solution 
satisfies all side constraints which are fulfilled by the PDE solution. 
\begin{lemma}\label{lma3.1}
Let $\{(\bu_h^{n}, \xi_h^{n}, \eta_h^{n})\}_{n\geq 0}$ be defined by the (FEA), then there hold
\begin{alignat}{2}
(\eta^{n}_h, 1)&=C_\eta(t_n) &&\qquad\mbox{for } n=0, 1, 2, \cdots,\label{e3.12}\\
(\xi^{n}_h, 1)&=C_\xi(t_{n-1+\theta}) &&\qquad\mbox{for }  n=1-\theta, 1, 2, \cdots, \label{e3.11} \\
\langle\bu^{n}_h\cdot\bn, 1\rangle &=C_{\bu}(t_{n-1+\theta})
&&\qquad\mbox{for } n=1-\theta, 1, 2, \cdots.\label{e3.13}
\end{alignat}
\end{lemma}

\begin{proof}
Taking $\psi_h=1$ in \eqref{e3.4} yields
\begin{align*}
\bigl(d_t\eta_h^{n+1}, 1\bigr) =(\phi,1) + \langle \phi_1, 1\rangle .
\end{align*}
Then summing over $n$ from $0$ to $\ell \,(\geq 0)$ we get
\begin{align*}
(\eta_h^{\ell+1},1)=(\eta_h^{0},1) + \bigl[(\phi,1) + \langle \phi_1, 1\rangle\bigr] t_{\ell+1}
=(\eta_0,1) + \bigl[(\phi,1) + \langle \phi_1, 1\rangle\bigr] t_{\ell+1}
=C_{\eta}(t_{\ell+1}) 
\end{align*}
for $\ell=0,1,2,\cdots.$ So \eqref{e3.12} holds.

To prove \eqref{e3.11}, taking $\bv_h=\bx$ in \eqref{e3.1} and $\varphi_h=1$ in \eqref{e3.2}, we get
\begin{align}
\mu \bigl( \div \bu_h^{n+1}, 1\bigr) -d\bigl( \xi_h^{n+1}, 1\bigr) 
&=\bigl( \bbf, x\bigr)+ \langle \bbf_1,x\rangle,\label{add_1}\\
\kappa_3\bigl(\xi_h^{n+1}, 1\bigr) +\bigl( \div \bu_h^{n+1}, 1\bigr) &=\kappa_1 C_\eta(t_{n+\theta}). \label{add_2}
\end{align}
Substituting \eqref{add_2} into \eqref{add_1} yields
\[
(d+\mu \kappa_3) \bigl( \xi_h^{n+1}, 1\bigr) = \mu\kappa_1 C_\eta(t_{n+\theta}) 
-\bigl( \bbf, x\bigr)- \langle \bbf_1,x\rangle.
\]
Hence, by the definition of $C_\xi(t)$ we conclude that \eqref{e3.11} holds for all $n\geq 1-\theta$. 

\eqref{e3.13} follows from \eqref{e3.12}, \eqref{e3.11}, \eqref{add_2}, and
an application of the Divergence Theorem. The proof is complete.
\end{proof}

The next lemma establishes an identity which mimics the continuous energy law for the solution 
of (FEA).

\begin{lemma}\label{lma3.2}
Let $\{(\bu_h^{n}, \xi_h^{n}, \eta_h^{n})\}_{n\geq 0}$ be defined by  (FEA),
then there holds the following identity:
\begin{align}\label{discrete_energy}
J_{h,\theta}^{\ell} +  S_{h,\theta}^{\ell} =J_{h,\theta}^0 \qquad \mbox{for } \ell\geq 1,\,\, \theta=0,1,
\end{align}
where
\begin{align*}
&J_{h,\theta}^\ell :=\frac{1}{2} \bigg[\mu\|\vepsi(\bu^{\ell+1}_h)\|_{L^2(\Ome)}^2
+\kappa_2\|\eta_h^{\ell+\theta}\|_{L^2(\Ome)}^2+\kappa_3\|\xi_h^{\ell+1}\|_{L^2(\Ome)}^2  
-2\bigl(\bbf, \bu^{\ell+1}_h \bigr)-2\bigl\langle \bbf_1,\bu^{\ell+1}_h \bigr\rangle\bigg],\\
&S_{h,\theta}^\ell := \Delta t\sum_{n=1}^{\ell}\Bigg[\frac{\mu\Delta t}{2}\|d_t\vepsi(\bu^{n+1}_h)\|_{L^2(\Ome)}^2
+\frac{K}{\mu_f} \bigl(\nab p_h^{n+1}-\rho_f\bg, \nab p_h^{n+1} \bigr)\\
&\hskip 0.5in
+\frac{\kappa_2\Delta t}{2}\|d_t\eta_h^{n+\theta}\|_{L^2(\Ome)}^2
+\frac{\kappa_3\Delta t}{2}\|d_t\xi_h^{n+1}\|_{L^2(\Ome)}^2-(\phi, p_h^{n+1})-\langle \phi_1,p_h^{n+1}\rangle \\
&\hskip 0.5in
-(1-\theta)\frac{\kappa_1K \Delta t}{\mu_f}\bigl(d_t\nab\xi_h^{n+1},\nab p_h^{n+1}\bigr)  \Bigg].\\
&p^{n+1}_h:=\kappa_1\xi^{n+1}_h +\kappa_2\eta^{n+\theta}_h.
\end{align*}
\end{lemma}

\begin{proof}
Since the proof for the case $\theta=1$ is exactly same as that of the PDE energy law, so we omit it and 
leave it to the interested reader to explore. Here we only consider the case $\theta=0$. Based on \eqref{e3.2}, we can define $\eta_h^{-1}$ by
\begin{align}
\kappa_1\bigl( \eta^{-1}_h, \varphi_h \bigr)=
\kappa_3\bigl(\xi^{0}_h, \varphi_h \bigr) +\bigl(\div\bu^{0}_h, \varphi_h \bigr)
 \end{align}
Setting $\bv_h=d_t\bu_h^{n+1}$ in \eqref{e3.1}, $\varphi_h=\xi_h^{n+1}$ in \eqref{e3.2}, 
and $\psi_h= p_h^{n+1}$ in \eqref{e3.4} after lowing the degree from $n+1$ to $n$, we get  
\begin{align}\label{add_3}
&\frac{\mu}{2}d_t \|\vepsi(\bu^{n+1}_h)\|_{L^2(\Ome)}^2 
+\frac{\mu}{2}\Delta t\|d_t\vepsi(\bu^{n+1}_h)\|_{L^2(\Ome)}^2 \\
&\hskip 1.2in
= d_t(\bbf, \bu^{n+1}_h)+d_t\langle \bbf_1,\bu^{n+1}_h\rangle+(\xi_h^{n+1},\div d_t\bu_h^{n+1}),\nonumber\\
&\kappa_3\bigl(d_t\xi^{n+1}_h, \xi_h^{n+1} \bigr) +\bigl(\div d_t\bu^{n+1}_h, \xi_h^{n+1} \bigr) 
=\kappa_1 \bigl( d_t\eta^{n}_h, \xi_h^{n+1} \bigr),  \label{add_4}\\
&\bigl(d_t\eta^{n}_h, p_h^{n+1} \bigr)
+\frac{1}{\mu_f} \bigl(K(\nab (\kappa_1\xi^{n}_h 
+\kappa_2\eta^{n}_h)-\rho_f\bg), \nab p_h^{n+1} \bigr)\label{add_5}\\
&\hskip 1.2in
=(\phi, p_h^{n+1})+\langle \phi_1,p_h^{n+1}\rangle\notag.
\end{align}

The first term on the left-hand side of \eqref{add_5} can be rewritten as
\begin{align}\label{add_6}
\bigl(d_t\eta^{n}_h, p_h^{n+1} \bigr) &= \bigl(d_t\eta^{n}_h, \kappa_1\xi^{n+1}_h+\kappa_2 \eta^n_h \bigr) \\
&=\kappa_1 \bigl(d_t\eta^{n}_h, \xi^{n+1}_h \bigr)
+\frac{\kappa_2\Delta t}{2}\|d_t\eta_h^{n}\|_{L^2(\Ome)}^2 +\frac{\kappa_2}{2}d_t\|\eta_h^{n}\|_{L^2(\Ome)}^2. \no
\end{align}
Moreover,
\begin{align}\label{add_6a}
&\frac{K}{\mu_f} \bigl(\nab (\kappa_1\xi^{n}_h 
+\kappa_2\eta^{n}_h)-\rho_f\bg, \nab p_h^{n+1})\\
&\hskip 0.5in
=\frac{K}{\mu_f}(\nab p_h^{n+1}-\rho_f\bg,\nab p_h^{n+1})
-\frac{\kappa_1 K\Delta t}{\mu_f}\bigl(d_t\nab\xi_h^{n+1},\nab p_h^{n+1}\bigr).\notag \\
&\kappa_3\bigl(d_t\xi_h^{n+1},\xi_h^{n+1} \bigr) 
=\frac{\kappa_3}{2} d_t\| \xi_h^{n+1}\|_{L^2(\Ome)}^2 
+ \frac{\kappa_3\Delta t}{2} \|d_t\xi_h^{n+1}\|_{L^2(\Ome)}^2. \label{add_6b}
\end{align}

Adding \eqref{add_3}--\eqref{add_5}, using \eqref{add_6}--\eqref{add_6b} and applying the summation 
operator $\Delta t\sum_{n=1}^{\ell}$ to the both sides of the resulting equation yield the desired 
equality \eqref{discrete_energy}.  The proof is complete.
\end{proof}

In the case $\theta=1$, \eqref{discrete_energy} gives the desired solution estimates without
any mesh constraint. On the other hand, when $\theta=0$, since the last term in the expression of
$S^\ell_{h,\theta}$ does not have a fixed sign, hence, it needs to be controlled in order to 
ensure the positivity of $S^\ell_{h,\theta}$.

\begin{corollary}\label{cor_discrete_energy}
Let $\{(\bu_h^{n}, \xi_h^{n}, \eta_h^{n})\}_{n\geq 0}$ be defined by  (FEA) with $\theta=0$, 
then there holds the following inequality:
\begin{align}\label{20141126_add_0}
J_{h,0}^{\ell} +  \widehat{S}_{h,0}^{\ell} \leq J_{h,0}^0 \qquad \mbox{for } \ell\geq 1,
\end{align}
provided that $\Delta t=O(h^2)$.  Where
\begin{align*}
\widehat{S}_{h,0}^\ell := \Delta t\sum_{n=1}^{\ell}\Bigg[\frac{\mu\Delta t}{4}\|d_t\vepsi(\bu^{n+1}_h)\|_{L^2(\Ome)}^2
+\frac{K}{2\mu_f}\|\nab p_h^{n+1}\|_{L^2(\Ome)}^2
-\frac{K}{\mu_f} \bigl(\rho_f\bg, \nab p_h^{n+1} \bigr)\\
\hskip 0.5in
+\frac{\kappa_2\Delta t}{2}\|d_t\eta_h^{n}\|_{L^2(\Ome)}^2
+\frac{\kappa_3\Delta t}{2}\|d_t\xi_h^{n+1}\|_{L^2(\Ome)}^2-(\phi, p_h^{n+1})-\langle \phi_1,p_h^{n+1}\rangle\Bigg].
\end{align*}
\end{corollary}

\begin{proof}
By Schwarz inequality and inverse inequality \eqref{e3.13-00}, we get
\begin{align}\label{20141126_add_1}
\frac{\kappa_1K \Delta t}{\mu_f}\bigl(d_t\nab\xi_h^{n+1},\nab p_h^{n+1}\bigr)
&\leq\frac{\kappa_1^2K}{2\mu_f}\|\nab\xi_h^{n+1}-\nab\xi_h^{n}\|_{L^2(\Ome)}^2
+\frac{K}{2\mu_f}\|\nab p_h^{n+1}\|_{L^2(\Ome)}^2\\
&\leq\frac{c_1^2\kappa_1^2K}{2\mu_fh^2}\|\xi_h^{n+1}-\xi_h^{n}\|_{L^(\Ome)2}^2
+\frac{K}{2\mu_f}\|\nab p_h^{n+1}\|_{L^2(\Ome)}^2\no
\end{align}
To bound the first term on the right-hand side of \eqref{20141126_add_1}, we appeal to the inf-sup condition and get
\begin{align}\label{20141126_add_2}
\|\xi_h^{n+1}-\xi_h^{n}\|_{L^2}&\leq\frac{1}{\beta_1}\sup_{v_h\in \bV_h}\frac{\bigl(\div\bv_h,\xi_h^{n+1}
-\xi_h^{n}\bigr)}{\|\nabla \bv_h\|_{L^2(\Ome)}}\\
&\leq\frac{\mu}{\beta_1}\sup_{v_h\in \bV_h}\frac{\bigl(\varepsilon(\bu^{n+1}-\bu^{n}),\varepsilon(\bv_h) 
 \bigr)}{\|\nabla \bv_h\|_{L^2(\Ome)}}\no\\
&\leq\frac{\mu}{\beta_1}\Delta t\|d_t\varepsilon(\bu_h^{n+1})\|_{L^2}.\no
\end{align}

Substituting \eqref{20141126_add_2} into \eqref{20141126_add_1} and combining it with 
\eqref{discrete_energy} imply \eqref{20141126_add_0} provided that 
$\Delta t \leq (\mu_f\beta_1^2)(2\mu K c_1^2\kappa_1^2)^{-1} h^2$.  The proof is complete.
\end{proof}

\subsection{Convergence analysis}\label{sec-3.4}
The goal of this section is to analyze the fully discrete finite element algorithm (FEA) proposed
in the previous subsection. Precisely, we shall derive optimal order error estimates for 
(FEA) in both $L^\infty(0,T;L^2(\Ome)$ and $L^2(0,T; H^1(\Ome))$-norm. To the end, we first list
some facts, which are well known in the literature \cite{bs08, brezzi}, about finite element functions.

We first recall the following inverse inequality for polynomial
functions \cite{cia}:
\begin{alignat}{2}\label{e3.13-00}
\|\nabla\varphi_h\|_{L^2(K)}\leq c_1h^{-1} \|\varphi_h\|_{L^2(K)} 
\qquad \forall\varphi_h\in P_r(K), K\in T_h.
\end{alignat}

For any $\varphi\in L^2(\Omega)$, we define its $L^2$-projection $\mathcal{Q}_h: L^2\rightarrow W_h$ as
\begin{align}\label{e3.13-01}
\bigl( \mathcal{Q}_h\varphi, \psi_h  \bigr)=\bigl( \varphi, \psi_h  \bigr) \qquad \psi_h\in W_h.
\end{align}

It is well known that the projection operator $\mathcal{Q}_h: L^2\rightarrow W_h$ satisfies 
(cf \cite{bs08}), for any $\varphi\in H^s(\Omega) (s\geq1)$,
\begin{align}\label{e3.13-02}
\|\mathcal{Q}_h\varphi-\varphi\|_{L^2(\Ome)}+h\| \nabla(\mathcal{Q}_h\varphi
-\varphi) \|_{L^2(\Ome)}\leq Ch^\ell\|\varphi\|_{H^\ell(\Ome)}, \quad \ell=\min\{2, s\}.
\end{align}

We like to point out that when $W_h\notin H^1(\Omega)$, the second term on the left-hand side 
of \reff{e3.13-02} has to be replaced by the broken $H^1$-norm.

Next, for any $\varphi\in H^1(\Omega)$, we define its elliptic projection $\mathcal{S}_h\varphi$ by
\begin{alignat}{2}\label{e3.13-03}
\bigl(K\nabla\mathcal{S}_h\varphi, \nabla\varphi_h\bigr) &=\bigl(K\nabla\varphi, \nabla\varphi_h\bigr)
&& \quad \varphi_h\in W_h,\\
\bigl(\mathcal{S}_h\varphi, 1\bigr) &=\bigl(\varphi, 1\bigr).&&\label{e3.13-04}
\end{alignat}

It is well known that the projection operator $\mathcal{S}_h: H^1(\Omega)\rightarrow W_h$ 
satisfies (cf \cite{bs08}), for any $\varphi\in H^s(\Omega) (s>1)$,
\begin{align}\label{e3.13-05}
\|\mathcal{S}_h\varphi-\varphi\|_{L^2(\Ome)}+h\| \nabla(\mathcal{S}_h\varphi-\varphi) \|_{L^2(\Ome)}
\leq Ch^\ell\|\varphi\|_{H^\ell(\Ome)}, \quad \ell=\min\{2, s\}.
\end{align}

Finally, for any $\bv\in \bH^1_\perp(\Omega)$, we define its elliptic projection $\mathcal{R}_h\bv$ by
\begin{alignat}{2}\label{e3.13-06}
\bigl(\varepsilon(\mathcal{R}_h\bv), \varepsilon(\bw_h)\bigr)
=\bigl(\varepsilon(\bv), \varepsilon(\bw_h)\bigr) \quad \bw_h\in \bV_h.
\end{alignat}

It is easy to show that the projection $\mathcal{R}_h\bv$ satisfies (cf \cite{bs08}), for any 
$\bv\in \bH^1_\perp(\Omega)\cap \bH^s(\Omega) (s>1)$,
\begin{align}\label{e3.13-07}
\|\mathcal{R}_h\bv-\bv\|_{L^2(\Ome)}+h\| \nabla(\mathcal{R}_h\bv-\bv) \|_{L^2(\Ome)}
\leq Ch^m\|\bv\|_{H^m(\Ome)}, \quad m=\min\{3, s\}.
\end{align}

To derive error estimates, we introduce the following error notation
\begin{alignat*}{2}
E_{\bu}^n &:=\bu(t_n)-\bu_h^n, &&\qquad E_\xi^n:=\xi(t_n)-\xi_h^n,
\qquad E_\eta^n:=\eta(t_n)-\eta_h^n, \\
E_p^n &:=p(t_n)-p_h^n, &&\qquad E_q^n:=q(t_n)-q_h^n. 
\end{alignat*}
It is easy to check that
\begin{alignat}{2}\label{e3.14}
E_p^n=\kappa_1E_\xi^n+\kappa_2E_\eta^n, \qquad
E_q^n=\kappa_3E_\xi^n+\kappa_1E_\eta^n.
\end{alignat}

Also, we denote
\begin{align*}
E_{\bu}^n &=\bu(t_n)-\mathcal{R}_h(\bu(t_n))+\mathcal{R}_h(\bu(t_n))-\bu_h^n
:=\Lambda_{\bu}^n+\Theta_{\bu}^n,\\ 
E_\xi^n &=\xi(t_n)-\mathcal{S}_h(\xi(t_n)) +\mathcal{S}_h(\xi(t_n))-\xi_h^n
:=\Lambda_{\xi }^n+\Theta_{\xi }^n,\\
E_\eta^n &=\eta(t_n)-\mathcal{S}_h(\eta(t_n))+\mathcal{S}_h(\eta(t_n))-\eta_h^n
:=\Lambda_{\eta }^n+\Theta_{\eta }^n,\\
E_p^n &=p(t_n)-\mathcal{Q}_h(p(t_n))+\mathcal{Q}_h(p(t_n))-p_h^n
:=\Lambda_{p }^n+\Theta_{p }^n.
\end{align*}

\begin{lemma}\label{lma3.3}
Let $\{ (\bu_h^n, \xi_h^n, \eta_h^n) \}_{n\geq0}$ be generated by the (FEA) and 
$\Lambda_{\bu}^n, \Theta_{\bu}^n, \Lambda_{\xi }^n, \Theta_{\xi }^n, \Lambda_{\eta }^n$ 
and $\Theta_{\eta }^n$ be defined as above. Then there holds the following identity:
\begin{align}\label{e3.15-00}
&\mathcal{E}_h^\ell +\Delta t\sum_{n=1}^\ell \Bigl[\frac{K}{\mu_f} 
\bigl( \nab\hat{\Theta}_p^{n+1}-\rho_f\bg,\hat{\Theta}_p^{n+1} \bigr)\\
&\qquad
+\frac{\Delta t}{2} \Bigl(\mu\| d_t\varepsilon(\Theta_{\bu}^{n+1}) 
\|_{L^2(\Ome)}^2+\kappa_2\| d_t\Theta_\eta^{n+\theta} \|_{L^2(\Ome)}^2
+\kappa_3\| d_t\Theta_\xi^{n+1} \|_{L^2(\Ome)}^2 \Bigr) \Bigr] \no\\
&=\mathcal{E}_h^0+\Delta t\sum_{n=1}^\ell \Bigl[\bigl( \Lambda_\xi^{n+1}, \div d_t \Theta_{\bu}^{n+1}\bigr)
-\bigl( \div d_t \Lambda_{\bu}^{n+1}, \Theta_\xi^{n+1} \bigr) \Bigr]  \no\\
&\qquad
+(\Delta t)^2\sum_{n=1}^\ell\bigl( d_t^2 \eta_h(t_{n+1}), \Theta_\xi^{n+1} \bigr)
+\Delta t\sum_{n=1}^\ell \bigl( R_h^{n+1}, \hat{\Theta}_p^{n+1} \bigr)  \no\\
&\qquad
+(1-\theta)(\Delta t)^2 \sum_{n=1}^\ell \frac{K\kappa_1}{\mu_f} \bigl( d_t\nabla\Theta_{\xi}^{n+1}, 
\nabla\hat{\Theta}_p^{n+1} \bigr),\no
\end{align}
where
\begin{alignat}{2}\label{e3.15-01}
\hat{\Theta}_p^{n+1} &:=\kappa_1\nabla\Theta_{\xi}^{n+1}+\kappa_2\nabla\Theta_{\eta}^{n+\theta}\\
\mathcal{E}_h^\ell &:=\frac{1}{2} \Bigl[\mu\|\varepsilon(\Theta_{\bu}^{\ell+1})\|_{L^2(\Ome)}^2
+ \kappa_2\|\Theta_\eta^{\ell+\theta} \|_{L^2(\Ome)}^2+\kappa_3\|\Theta_\xi^{\ell+1}\|_{L^2(\Ome)}^2 \Bigr],\\
R_h^{n+1} &:=-\frac{1}{\Delta t}\int_{t_{n}}^{t_{n+1}}(s-t_{n})\eta_{tt}(s)\,ds.\label{e3.15-02}
\end{alignat}
\end{lemma}

\begin{proof}
Subtracting \reff{e3.1} from \reff{e2.4}, \reff{e3.2} from \reff{e2.5},  \reff{e3.4} from \reff{e2.6}, 
respectively, we get the following error equations:
\begin{alignat}{2}\label{e3.15}
&\mu \bigl(\vepsi(\bE^{n+1}_{\bu}), \vepsi(\bv_h) \bigr)-\bigl( E_\xi^{n+1}, \div \bv_h \bigr)= 0
&&\quad \forall \bv_h\in V_h, \\
&\kappa_3\bigl(E_{\xi}^{n+1},\varphi_h\bigr)+\bigl(\div\bE^{n+1}_{\bu}, \varphi_h \bigr) && \label{e3.16} \\
&\hskip 0.9in 
= \kappa_1\bigl( E^{n+\theta}_{\eta}, \varphi_h \bigr)+\Delta t\bigl(  d_t\eta(t_{n+1}), \varphi_h  \bigr)
&&\quad \forall \varphi_h \in M_h, \no \\
&\bigl(d_tE_\eta^{n+1}, \psi_h \bigr)
+\frac{K}{\mu_f} \bigl(\nab (\kappa_1E_\xi^{n+1} +\kappa_2E_\eta^{n+1})-\rho_f\bg, \nab \psi_h \bigr)
&&  \label{e3.18}\\
&\hskip 1.9in
=(R^{n+1}_h, \psi_h) &&\quad\forall \psi_h \in W_h,\no\\
&  E_{\bu}^0=0,\quad E_\xi^0=0, \quad E_\eta^{-1}=0. &&\label{e3.19}
\end{alignat}

Using the definition of the projection operators $\mathcal{Q}_h, \mathcal{S}_h, \mathcal{R}_h$, we have 
\begin{alignat}{2}\label{e3.20}
&\mu \bigl(\vepsi(\Theta^{n+1}_{\bu}), \vepsi(\bv_h) \bigr)-\bigl( \Theta_\xi^{n+1}, \div \bv_h \bigr)
= (\Lambda_\xi^{n+1}, \div \bv_h),
&&\quad\forall \bv_h\in V_h, \\
&\kappa_3\bigl(\Theta_\xi^{n+1}, \varphi_h \bigr) +\bigl(\div\Theta^{n+1}_{\bu}, \varphi_h \bigr) 
=\kappa_1 \bigl(\Theta_\eta^{n+\theta}, \varphi_h \bigr) \label{e3.21}\\
&\hskip 0.8in
-\bigl( \div \Lambda_{\bu}^{n+1}, \varphi_h \bigr)
+\Delta t\bigl(  d_t\eta(t_{n+1}), \varphi_h  \bigr)
&&\quad\forall \varphi_h \in M_h,\notag \\
&\bigl(d_t\Theta_\eta^{n+1}, \psi_h \bigr) +\frac{K}{\mu_f} \bigl(\nab (\kappa_1\Theta_\xi^{n+1} 
+\kappa_2\Theta_\eta^{n+1})-\rho_f\bg, \nab \psi_h \bigr)  &&\label{e3.23}\\
&\hskip 1.5in
=\bigl( R^{n+1}_h, \psi_h \bigr) &&\quad\forall \psi_h \in W_h, \no\\
&E_{\bu}^0=0,\quad E_\xi^0=0, \quad E_\eta^{-1}=0. &&\label{e3.24}
\end{alignat}

\eqref{e3.15-00} follows from setting $\bv_h=d_t\Theta^{n+1}_{\bu}$ in \reff{e3.20}, $\varphi_h=\Theta^{n+1}_\xi$ 
(after applying the difference operator $d_t$ to the equation \reff{e3.21}), 
$\psi_h=\hat{\Theta}_p^{n+1}=\kappa_1\Theta_\xi^{n+1} +\kappa_2\Theta_\eta^{n+\theta}$ in \reff{e3.23},
adding the resulting equations, and applying the summation operator $\Delta t\sum^{\ell}_{n=1}$ to both sides. 
\end{proof}

\begin{theorem}\label{thm1301}
Let $\{(u_h^{n}, \xi_h^{n}, \eta_h^{n})\}_{n\geq 0}$ be defined by (FEA), then 
there holds the error estimate for $\ell \leq N$
\begin{align}\label{e131130-1}
&\max_{0\leq n\leq \ell}\bigg[ \sqrt{\mu}\|\varepsilon(\Theta_{\bu}^{n+1})\|_{L^2(\Ome)}
+\sqrt{\kappa_2}\|\Theta_\eta^{n+\theta}\|_{L^2(\Ome)}+\sqrt{\kappa_3}\|\Theta_\xi^{n+1}\|_{L^2(\Ome)}\bigg] \\
&\hskip 1.2in
+\bigg[\Delta t\sum_{n=0}^\ell\frac{K}{\mu_f} 
\|\hat{\Theta}_p^{n+1}\|_{L^2}^2\bigg]^\frac{1}{2} 
\leq C_1(T)\Delta t+C_2(T)h^2,\no
\end{align}
provided that $\Delta t =O(h^2)$ when $\theta=0$ and $\Delta t >0$ when $\theta=1$. Where 
\begin{align}
C_1(T)&= C\|q_t\|_{L^2((0, T); L^2(\Ome))}^2+ C\|(q)_{tt}\|_{L^2((0, T); H^{-1}(\Ome))},\\
C_2(T)&=C\|\xi\|_{L^{\infty}((0,T);H^2(\Ome))} +C\|\xi_t\|_{L^2((0,T);H^2(\Ome))} \\
&\hskip 1.4in
+C\|\div({\bu})_t\|_{L^2((0,T);H^2(\Ome))}. \no
\end{align}
\end{theorem}

\begin{proof}
To derive the above inequality, we need to bound each term on the right-hand side of \reff{e3.15-00}. 
Using the fact $\Theta_{\bu}^0=\mathbf{0}$, $\Theta_\xi^0=0$ and $\Theta_\eta^{-1}=0$, we have
\begin{align}\label{e131204-00}
&\mathcal{E}_h^\ell +\Delta t\sum_{n=1}^\ell\bigg[\frac{K}{\mu_f} \bigl(\nabla\hat{\Theta}_p^{n+1}-\rho_f\bg, 
\nabla\hat{\Theta}_p^{n+1} \bigr)&&\\
&\qquad
+\frac{\Delta t}{2}\Bigl(\mu\| d_t\varepsilon(\Theta_{\bu}^{n+1}) \|_{L^2(\Ome)}^2
+\kappa_2\| d_t\Theta_\eta^{n+\theta} \|_{L^2(\Ome)}^2+\kappa_3\| d_t\Theta_\xi^{n+1} \|_{L^2(\Ome)}^2 \Bigr) \bigg]\no\\
&=(\Delta t)^2\sum_{n=1}^\ell\bigl( d_t^2 \eta(t_{n+1}), \Theta_\xi^{n+1} \bigr)
+\Delta t\sum_{n=1}^\ell \bigl( R_h^{n+1}, \hat{\Theta}_p^{n+1} \bigr)
+\frac{\mu}{2} \|\varepsilon(\Theta_{\bu}^{1})\|_{L^2(\Ome)}^2 \no\\
&\qquad
+\Delta t\sum_{n=1}^\ell \Bigl[\bigl( \Lambda_\xi^{n+1}, \div d_t \Theta_{\bu}^{n+1}\bigr)
-\bigl( \div d_t \Lambda_{\bu}^{n+1}, \Theta_\xi^{n+1} \bigr) \Bigr] \no\\
&\qquad
+(1-\theta)(\Delta t)^2\sum_{n=1}^\ell\frac{K\kappa_1}{\mu_f} \bigl(d_t\nab\Theta_{\xi}^{n+1}, \nab\hat{\Theta}_p^{n+1}\bigr),
\no
\end{align}

We now estimate each term on the right-hand side of \reff{e131204-00}. To bound the first term on the 
right-hand side of \reff{e131204-00}, we first use the summation by parts formula and $d_t\eta_h(t_0)=0$ to get
\begin{align}\label{e131209-00}
\sum_{n=0}^\ell \bigl( d_t^2 \eta(t_{n+1}), \Theta_\xi^{n+1} \bigr)
=\frac{1}{\Delta t}\bigl( d_t \eta(t_{l+1}),  \Theta_\xi^{l+1}\bigr)
-\sum_{n=1}^\ell \bigl( d_t \eta(t_{n}), d_t\Theta_\xi^{n+1} \bigr).
\end{align}

Now, we bound each term on the right-hand side of \reff{e131209-00} as follows:
\begin{align}\label{e131209-01}
&\frac{1}{\Delta t}\bigl( d_t \eta(t_{\ell+1}), \Theta_\xi^{\ell+1}\bigr)
\leq\frac{1}{\Delta t}\|d_t \eta(t_{\ell+1})\|_{L^2(\Ome)}\| \Theta_\xi^{\ell+1}\|_{L^2(\Ome)}\\
&\quad
\leq \frac{1}{\Delta t}\|\eta_t\|_{L^2((t_\ell, t_{\ell+1}); L^2(\Ome))}\cdot \frac{1}{\beta_1}
\sup_{\bv_h\in\bV_h}\frac{\mu\bigl( \varepsilon(\Theta_{\bu}^{\ell+1}), \varepsilon(\bv_h)  \bigr)
-(\Lambda_{\xi}^{\ell+1},\div \bv_h)}{\|\nabla \bv_h\|_{L^2(\Ome)}}\no\\
&\quad
\leq \frac{C\mu}{\beta_1\Delta t}\|\eta_t\|_{L^2((t_\ell, t_{\ell+1}); L^2(\Ome))}
\Bigl[\|\varepsilon(\Theta_{\bu}^{\ell+1})\|_{L^2(\Ome)}+\frac{1}{\mu}\|\Lambda_{\xi}^{\ell+1}\|_{L^2(\Ome)}\Bigr]\no\\
&\quad
\leq \frac{\mu}{4(\Delta t)^2}  \|\varepsilon(\Theta_{\bu}^{\ell+1})\|_{L^2(\Ome)}^2
+\frac{C\mu}{\beta_1^{2}} \|\eta_t\|_{L^2((t_\ell, t_{\ell+1}); L^2(\Ome))}^2
+\frac{C}{\beta_1^{2}}\|\Lambda_{\xi}^{\ell+1}\|_{L^2(\Ome)}^2,\no\\
&\sum_{n=1}^\ell\bigl( d_t \eta(t_{n}), d_t\Theta_\xi^{n+1} \bigl)
\leq \sum_{n=1}^\ell\| d_t \eta(t_{n})\|_{L^2(\Ome)}\| d_t\Theta_\xi^{n+1}\|_{L^2(\Ome)}\\
&\quad
\leq \sum_{n=1}^\ell|| d_t \eta(t_{n})\|_{L^2(\Ome)}\cdot\frac{1}{\beta_1}
\sup_{\bv_h\in\bV_h}\frac{\mu\bigl( d_t\varepsilon(\Theta_{\bu}^{n+1}), \varepsilon(\bv_h)  \bigr)
-(d_t\Lambda_{\xi}^{n+1},\div \bv_h)}{\|\nabla \bv_h\|_{L^2(\Ome)}}\no\\
&\quad
\leq \frac{C\mu}{\beta_1} \sum_{n=1}^\ell \| d_t \eta(t_{n})\|_{L^2(\Ome)}
\Bigl[ \| d_t \varepsilon(\Theta_{\bu}^{n+1})\|_{L^2(\Ome)}+\frac{1}{\mu}\|d_t\Lambda_{\xi}^{n+1}\|_{L^2(\Ome)}  \Bigr]\no\\
&\quad
\leq \sum_{n=1}^\ell \Bigl[ \frac{\mu}{4}\| d_t \varepsilon(\Theta_{\bu}^{n+1})\|_{L^2(\Ome)}^2
+\frac{C}{\beta_1^2}\|d_t\Lambda_{\xi}^{n+1}\|_{L^2(\Ome)} \Bigr]
+\frac{C\mu}{\beta_1^2} \|\eta_t\|_{L^2((0,T);L^2(\Ome))}^2.\no
\end{align}
The second term on the right-hand side of \reff{e131204-00} can be bounded as
\begin{align}
\bigl|\bigl( R_h^{n+1}, \hat{\Theta}_p^{n+1} \bigr) \bigr|
&\leq \|R_h^{n+1}\|_{H^{-1}(\Ome)}\|\nabla\hat{\Theta}_p^{n+1}\|_{L^2(\Ome)} \label{e131209-02}\\
&\leq \frac{K}{4\mu_f}\|\nabla\hat{\Theta}_p^{n+1}\|_{L^2(\Ome)}^2
        +\frac{\mu_f}{K}\|R_h^{n+1}\|_{H^{-1}(\Ome)}^2\no\\
&\leq \frac{K}{4\mu_f}\|\nabla\hat{\Theta}_p^{n+1}\|_{L^2(\Ome)}^2
        +\frac{\mu_f\Delta t}{3K}\|\eta_{tt}\|_{L^2((t_{n}, t_{n+1}); H^{-1}(\Ome))}^2,\no
\end{align}
where we have used the fact that
\begin{align*} 
\|R_h^{n+1}\|_{H^{-1}(\Ome)}^2 \leq\frac{\Delta t}{3}\int_{t_{n}}^{t_{n+1}} \|\eta_{tt}\|_{H^{-1}(\Ome)}^2\,dt.
\end{align*}

The fourth term on the right-hand side  of \reff{e131204-00} can be bounded by
\begin{align}\label{e141104-01}
&\Delta t\sum_{n=1}^\ell \Bigl[ \bigl( \Lambda_\xi^{n+1}, \div d_t \Theta_{\bu}^{n+1}\bigr)
-\bigl( \div d_t \Lambda_{\bu}^{n+1}, \Theta_\xi^{n+1} \bigr) \Bigr]\\
&\quad
=( \Lambda_\xi^{\ell+1}, \div\Theta_{\bu}^{\ell+1}) -\Delta t\sum_{n=1}^\ell \Bigr[\bigl( d_t\Lambda_\xi^{n+1}, 
\div\Theta_{\bu}^{n}\bigr) 
+\bigl( \div d_t \Lambda_{\bu}^{n+1}, \Theta_\xi^{n+1} \bigr) \Bigr]\notag\\
&\quad
\leq \frac12\|\Lambda_\xi^{\ell+1}\|_{L^2(\Ome)}^2 +\frac12\|\div \Theta_{\bu}^{\ell+1}\|_{L^2(\Ome)}^2 
+\frac12\Delta t\sum_{n=1}^\ell \Bigl[ \|d_t\Lambda_\xi^{n+1}\|_{L^2(\Ome)}^2 \notag\\
&\hskip 0.6in
+ C\|\varepsilon(\Theta_{\bu}^{n+1})\|_{L^2(\Ome)}^2 + \|\div d_t \Lambda_{\bu}^{n+1}\|_{L^2(\Ome)}^2
+ \|\Theta_\xi^{n+1}\|_{L^2(\Ome)}^2 \Bigr],\notag
\end{align}
here we have used Korn's inequality 
\begin{align*}
\|\div \Theta_{\bu}^{n+1}\|_{L^2(\Ome)} \leq C\|\varepsilon(\Theta_{\bu}^{n+1})\|_{L^2(\Ome)}.
\end{align*}

When $\theta=0$ we also need to bound the last term on the right-hand side of \reff{e131204-00}, 
which is carried out below.
\begin{align}\label{e141104-02}
&\sum_{n=1}^\ell \bigl(d_t\nabla\Theta_{\xi}^{n+1}, \nabla\hat{\Theta}_p^{n+1} \bigr)
\leq\sum_{n=1}^\ell\|d_t\Theta_{\xi}^{n+1}\|_{L^2(\Ome)}\|\nabla\hat{\Theta}_p^{n+1}\|_{L^2(\Ome)}\\
&\,
\leq\sum_{n=1}^\ell \sup_{\bv_h\in \bV_h}
\frac{\mu\bigl( d_t\varepsilon(\Theta_{\bu}^{n+1}), \varepsilon(\bv_h)  \bigr)
-\bigl(d_t\Lambda_{\xi}^{n+1},\div \bv_h\bigl)}{\|\nabla \bv_h\|_{L^2(\Ome)}}
\,\|\nabla\hat{\Theta}_p^{n+1}\|_{L^2(\Ome)} \no \\
&\,
\leq \sum_{n=1}^\ell \Bigl[ \frac{\mu^2\kappa_1\Delta t}{h^2\beta_1^2} \|d_t\varepsilon(\Theta_{\bu}^{n+1})\|_{L^2}^2 
+\frac{\kappa_1\Delta t}{h^2\beta_1^2} \|d_t\Lambda_{\xi}^{n+1}\|_{L^2}
+\frac{\kappa_1^{-1}}{4\Delta t}\|\nabla\hat{\Theta}_p^{n+1}\|_{L^2(\Ome)}^2 \Bigr].\no
\end{align}

Substituting \reff{e131209-00}--\reff{e141104-02} into \reff{e131204-00} and rearranging terms we get
\begin{align}\label{e131209-04}
&\mu \|\varepsilon(\Theta_{\bu}^{\ell+1})\|_{L^2(\Ome)}^2
+\kappa_2\|\Theta_\eta^{\ell+\theta}\|_{L^2(\Ome)}^2+\kappa_3\|\Theta_\xi^{\ell+1}\|_{L^2(\Ome)}^2 \\
&\hskip 0.5in
+\Delta t \sum_{n=1}^\ell\frac{K}{\mu_f} \|\nabla\hat{\Theta}_p^{n+1} \|_{L^2(\Ome)}^2  \no \\
&\quad
\leq \frac{4\mu_f(\Delta t)^2}{\kappa} \|\eta_{tt}\|_{L^2((0, T); H^{-1})}^2
+\frac{4\mu(\Delta t)^2}{\beta_1^2} \|\eta_t\|_{L^2((0,T);L^2(\Ome))}^2, \no\\
&\hskip 0.5in
+\|\Lambda_\xi^{\ell+1}\|_{L^2(\Ome)}^2
+\Delta t\sum_{n=1}^\ell \|d_t\Lambda_\xi^{n+1}\|_{L^2(\Ome)}^2
+\Delta t\sum_{n=1}^\ell \|\div d_t \Lambda_{\bu}^{n+1}\|_{L^2(\Ome)}^2, \no
\end{align}
provide that  $\Delta t\leq \mu_f\beta_1^2(4\mu\kappa_1^2K)^{-1} h^2$ when $\theta=0$, but
it holds for all $\Delta t>0$ when $\theta=1$. Hence, \reff{e131130-1} follows from using the 
approximation properties of the projection operators $\mathcal{Q}_h, \mathcal{R}_h$ and 
$\mathcal{S}_h$. The proof is complete.
\end{proof}

We conclude this section by stating the main theorem of the section.

\begin{theorem}\label{thm3.5}
The solution of the (FEA) satisfies the following error estimates:
\begin{align}\label{e131209-05}
&\max_{0\leq n\leq N} \Bigl[ \sqrt{\mu} \|\nabla(u(t_n)-u_h^n)\|_{L^2(\Ome)}
+\sqrt{\kappa_2} \|\eta(t_n)-\eta_h^n\|_{L^2(\Ome)}\\
&\hskip .5in
+\sqrt{\kappa_3} \|\xi(t_n)-\xi_h^n\|_{L^2(\Ome)} \Bigr]
\leq \widehat{C}_1(T) \Delta t +\widehat{C}_2(T)h^2.\no \\
&\bigg[ \Delta t \sum_{n=0}^N \frac{K}{\mu_f} \|\nabla(p(t_n)-p_h^n) \|_{L^2(\Ome)}^2 \bigg]^{\frac12}
\leq \widehat{C}_1(T) \Delta t +\widehat{C}_2(T)h, \label{E_p}
\end{align}
provided that $\Delta t=O(h^2)$ when $\theta=0$ and $\Delta t>0$ when $\theta=1$. Where 
\begin{align*}
\widehat{C}_1(T)&:=C_1(T),\\
\widehat{C}_2(T)&:=C_2(T)+\|\xi\|_{L^{\infty}((0,T);H^2(\Ome))}+\|\eta\|_{L^{\infty}((0,T);H^2(\Ome))} \\
&\hskip 1.85in
+\|\nabla\bu\|_{L^{\infty}((0,T);H^2(\Ome))}.
\end{align*}
\end{theorem}

\begin{proof}
The above estimates follow immediately from an application of the triangle inequality on
\begin{alignat*}{2}
\bu(t_n)-\bu_h^n &=\Lambda_{\bu}^n+\Theta_{\bu}^n, 
&&\qquad \xi(t_n)-\xi_h^n =\Lambda_{\xi }^n+\Theta_{\xi }^n, \\
\eta(t_n)-\eta_h^n &=\Lambda_{\eta }^n+\Theta_{\eta }^n, 
&&\qquad p(t_n)-p_h^n =\hat{\Lambda}_p^n + \hat{\Theta}_p^n.
\end{alignat*}
and appealing to \reff{e3.13-02}, \reff{e3.13-07} and Theorem \ref{thm1301}.
\end{proof}

%% file: section_4a.tex
\section{Numerical experiments}\label{sec-4}
In this section we shall present three $2$-dimensional numerical experiments to validate theoretical 
results for the proposed numerical methods, to numerically examine 
the performances of the approach and methods as well as to compare them with existing methods
in the literature on two benchmark problems. One of these two problems was used 
to demonstrate the ``locking phenomenon" in \cite{pw09}. Our numerical result shows that
such a ``locking phenomenon" phenomenon does not occur in our numerical methods, it confirms 
the fact that our approach and methods have a built-in mechanism to prevent the ``locking phenomenon".

\medskip
{\bf Test 1.} Let $\Omega=[0,1]\times [0,1]$, $\Gamma_1=\{(1,x_2); 0\leq x_2\leq 1\}$, 
$\Gamma_2=\{(x_1,0); 0\leq x_1\leq 1\}$,  $\Gamma_3=\{(0,x_2); 0\leq x_2\leq 1\}$,
$\Gamma_4=\{(x_1,1); 0\leq x_1\leq 1\}$, and $T=0.001$. We consider problem \eqref{e2.6}--\eqref{e2.9} 
with following source functions:
\begin{align*}
\mathbf{f} &=-(\lambda+\mu)t(1,1)^T+\alpha \cos(x_1+x_2)e^t(1,1)^T,\\
\phi &=\Bigl(c_0+\frac{2\kappa}{\mu_f} \Bigr)\sin(x_1+x_2)e^t+\alpha(x_1+x_2),
\end{align*}
and the following boundary and initial conditions:
\begin{alignat*}{2}
p &= \sin(x_1+x_2)e^t  &&\qquad\mbox{on }\partial\Omega_T,\\
u_1 &= \frac12 x_1^2t &&\qquad\mbox{on }\Gamma_j\times (0,T),\, j=1,3,\\
u_2 &= \frac12 x_2^2t &&\qquad\mbox{on }\Gamma_j\times (0,T),\, j=,2,4,\\
\sigma\bf{n}-\alpha \emph{p}\bf{n} &= \mathbf{f}_1, &&\qquad \mbox{on } \p\Ome_T;\\
\mathbf{u}(x,0) = \mathbf{0},  \quad p(x,0) &=\sin(x_1+x_2) &&\qquad\mbox{in } \Ome,
\end{alignat*}
where
\[
\mathbf{f}_1(x,t)=\mu(x_1n_1,x_2n_2)^T t + \lambda(x_1+x_2) (n_1,n_2)^T t -\alpha\sin(x_1+x_2)(n_1,n_2)^T e^t.  \]
It is easy to check that the exact solution for this problem is  
\[
\mathbf{u}(x,t)=\frac{t}2 \bigl( x_1^2, x_2^2 \bigr)^T,\qquad p(x,t)=\sin(x_1+x_2)e^t.
\]
We note that the boundary conditions used above are not pure Neumann conditions, 
instead, they are mixed Dirichlet-Neumann conditions.  As pointed out in Remark \ref{rem-2.1} (c),
the approach and methods of this paper also apply to this case, the only change is to
replace the test and trial space $\bH^1_\perp(\Ome)$ by $\bH^1(\Ome)$ with some appropriately built-in 
Dirichlet boundary condition in Definition \ref{weak2}.

The goal of doing this test problem is to compute the order of the exact errors and to 
show that the theoretical error bounds proved in the previous section are sharp. 

Table \ref{tab1} displays the computed $L^\infty(0,T; L^2(\Ome))$ and $L^2(0,T;H^1(\Ome))$-norm errors 
and the convergence rates with respect to $h$ at the terminal time $T$. In the test,  
$\Delta t=10^{-5}$ is used so that the time error is negligible. Evidently, the spatial rates of 
convergence are consistent with that proved in the convergence theorem.   
\begin{table}[htb]
\begin{center}
\begin{tabular}{|l|c|c|c|c|}
\hline
& $L^\infty(L^2)$ error & $L^\infty(L^2)$ order& $L^2(H^1)$ error& $L^2(H^1)$ order\\ \hline
$h=0.16$ & 2.0789e-3 & & 5.5045e-2 & \\ \hline
$h=0.08$ & 5.9674e-4 &1.8006 &  2.9431e-2&0.9032 \\ \hline
$h=0.04$ & 1.6227e-4 &1.8787&  1.5332e-2&0.9408 \\ \hline
$h=0.02$ & 4.0971e-5 & 1.9857&  7.6968e-3& 0.9942\\ \hline
\end{tabular}
\smallskip
\caption{Spatial errors and convergence rates of Test 1.} \label{tab1}
\end{center}
\end{table}

Figures \ref{figure_p3} and \ref{figure_p4_add1} show respectively the surface plot of the computed 
pressure $p$ at the terminal time $T$ and the color plot of both the computed pressure $p$ and 
displacement $\mathbf{u}$ with mesh parameters $h=0.02$ and $\Delta t=10^{-5}$. They coincide with
the exact solution on the same space-time resolution.

\begin{figure}[th]
\centering
\includegraphics[height=2.5in,width=2.8in]{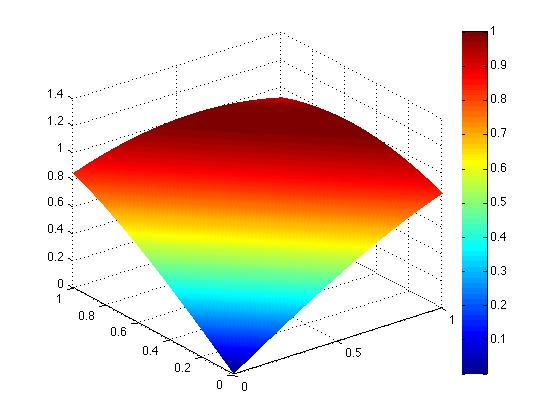}
\caption{Test 1: Surface plot of the Computed pressure $p$ at the terminal time $T$.}\label{figure_p3}
\end{figure}


\begin{figure}[th]
\centering
\includegraphics[height=2.5in,width=2.8in]{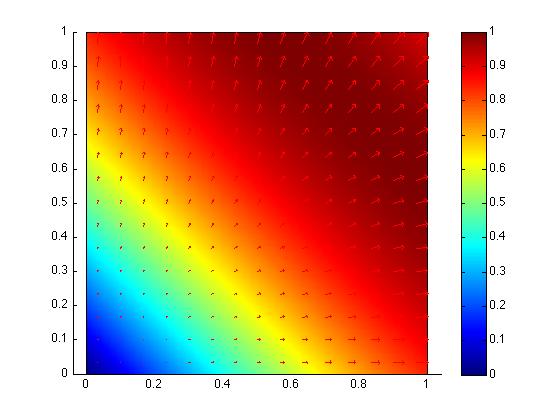}
\caption{Test 1: Computed pressure $p$ (color plot) and displacement $\mathbf{u}$ (arrow plot) at $T$.}
\label{figure_p4_add1}
\end{figure}

%
%

\medskip
{\bf Test 2.} In this test we consider so-called Barry-Mercer's problem, which is a Benchmark test 
problem for the poroelasticity model \eqref{e2.6}--\eqref{e2.9} (cf. \cite{pw07b,pw09} and the references 
therein). Again, $\Ome=[0,1]\times [0,1]$ but $T=1$. Barry-Mercer's problem assumes no source, that is, 
$\mathbf{f}\equiv 0$ and $\phi\equiv 0$, and takes the following boundary conditions:
\begin{alignat*}{2}
p &= 0 &&\qquad\mbox{on }\Gamma_j\times (0,T),\, j=1,3,4,\\
p &= p_2 &&\qquad\mbox{on }\Gamma_j\times (0,T),\, j=2,\\
u_1 &= 0 &&\qquad\mbox{on }\ \Gamma_j\times (0,T),\, j=1,3,\\
u_2 &= 0 &&\qquad\mbox{on }\ \Gamma_j\times (0,T),\, j=2,4,\\
\sigma\bf{n}-\alpha p\bf{n} &= \mathbf{f}_1 :=(0,\alpha p)^T &&\qquad\mbox{on }\partial\Omega_T,
\end{alignat*}
where
\[
p_2(x_1,t) 
=\begin{cases} 
\sin t &\quad\mbox{when } x\in [0.2,0.8)\times (0,T), \\
0 &\quad\mbox{others}. 
\end{cases}
\]

The boundary segments $\Gamma_j, j=1,2,3,4$, which are defined in {\bf Test 1}, 
and the above boundary conditions are depicted in Figure \ref{figure_p5}. Also, the initial conditions
for Barry-Mercer's problem are $\mathbf{u}(x,0) \equiv \mathbf{0}$ and $p(x,0)\equiv 0$.
We remark that Barry-Mercer's problem has a unique solution which is given by an infinite 
series (cf. \cite{pw09}).
\begin{figure}[th]
\centering
\includegraphics[height=2.5in,width=2.8in]{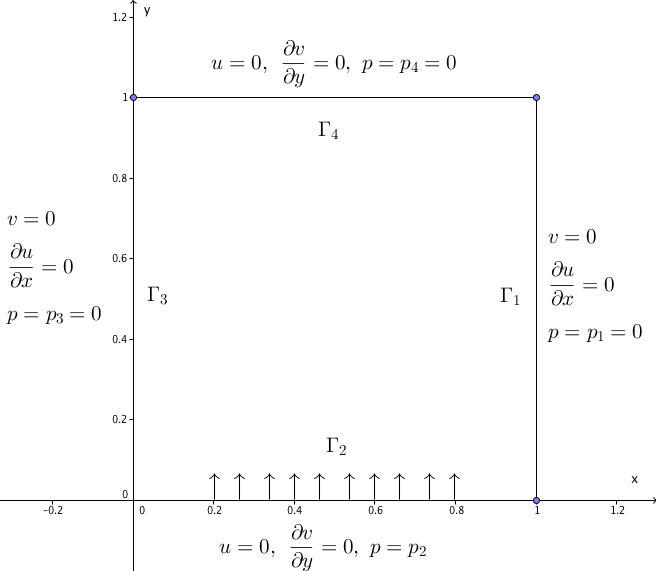}
\caption{Test 2: boundary conditions.}\label{figure_p5}
\end{figure}

Figures \ref{figure_p6} and \ref{figure_p8} display respectively the computed pressure $p$ (surface plot)
and the computed displacement $\mathbf{u}$ (arrow plot). We note that the arrows near the boundary 
match very well with those on the boundary. Our numerical solution approximates the 
exact solution of Barry-Mercer's problem very well and does not produce any oscillation 
in computed pressure.
\begin{figure}[th]
\centering
\includegraphics[height=2.5in,width=2.8in]{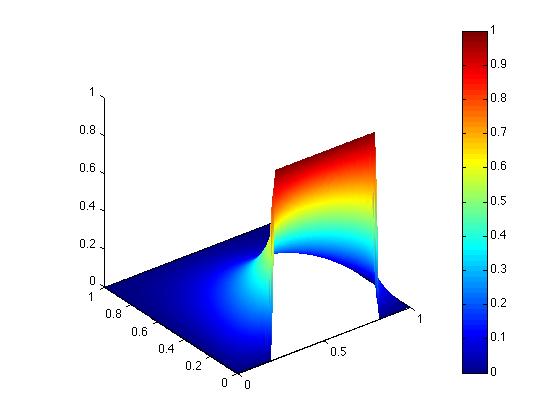}
\caption{Test 2: Surface plot of the computed pressure $p$ at the terminal time $T$.}\label{figure_p6}
\end{figure}


\begin{figure}[th]
\centering
\includegraphics[height=2.5in,width=2.8in]{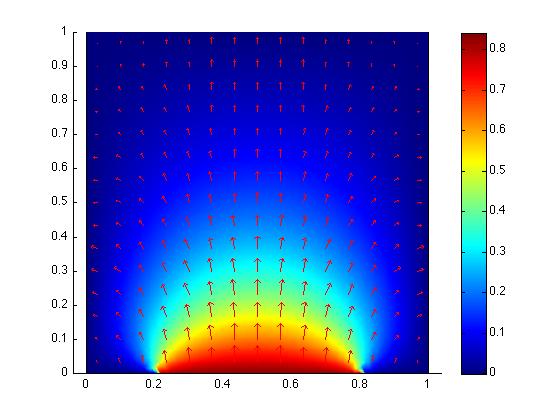}
\caption{Test 2: Computed pressure $p$ (color plot) and displacement (arrow plot) at $T$.}\label{figure_p8}
\end{figure}

\medskip
{\bf Test 3.} This test problem is taken from \cite{pw09}. Again, we consider problem 
\eqref{e2.6}--\eqref{e2.9} with $\Ome=[0,1]\times [0,1]$. Let $\Gamma_j$ be same as in {\bf Test 1}
and $c_0=0, E=10^5, \nu=0.4, \mu=35714$ and $T=0.001$. There is no source, that is,
$\mathbf{f}\equiv 0$ and $\phi\equiv 0$. The boundary conditions are taken as  
\begin{alignat*}{2}
-\frac{\kappa}{\mu_f}(\nabla p-\rho_f\mathbf{g})\cdot\mathbf{n} &= 0 &&\qquad \mbox{on }\p\Ome_T,\\
\mathbf{u} &=\mathbf{0} &&\qquad \mbox{on }\Gamma_3\times (0,T),\\
\sigma\bf{n}-\alpha p\bf{n} &= \mathbf{f}_1 &&\qquad\mbox{on }\Gamma_j\times (0,T),\, j-1,2,4,
\end{alignat*}
where  $\mathbf{f}_1=(f_1^1, f_1^2)$ and
\[
f_1^1\equiv 0 \quad \mbox{on }\p\Ome_T, \qquad
f_1^2 = \begin{cases}
      0 &\quad\mbox{on } \Gamma_j\times (0,T),\, j=1,2,3,\\
      -1 &\quad\mbox{on } \Gamma_4\times (0,T).
      \end{cases}
\]
The computational domain $\Ome$ and the above boundary conditions are depicted in Figure \ref{test1_1}.
Also, the zero initial conditions are assigned for both $\mathbf{u}$ and $p$ in this test.
\begin{figure}[th]
\centering
\includegraphics[height=2.5in,width=2.8in]{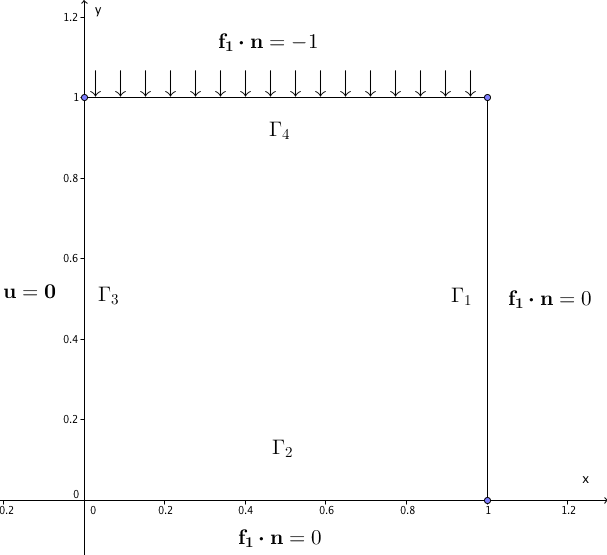}
\caption{Test 3: boundary conditions}\label{test1_1}
\end{figure}

Figures \ref{test1_2}--\ref{test1_4} display respectively the surface and color plot
of the computed pressure, the arrow plot of the displacement vector, and the deformation 
of the whole $\Ome$. There is no oscillation in the computed pressure and the arrows 
near the boundary match very well with arrows on the boundary.
\begin{figure}[th]
\centering
\includegraphics[height=2.4in,width=2.4in]{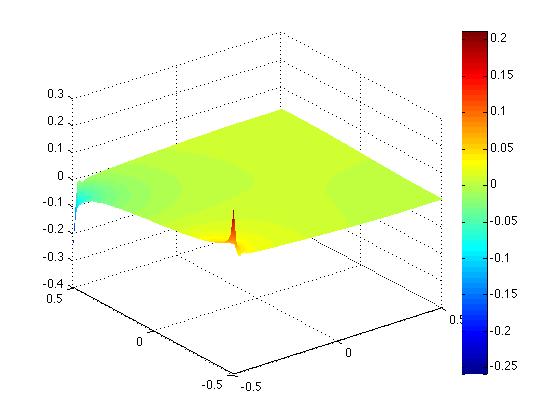}
\includegraphics[height=2.4in,width=2.4in]{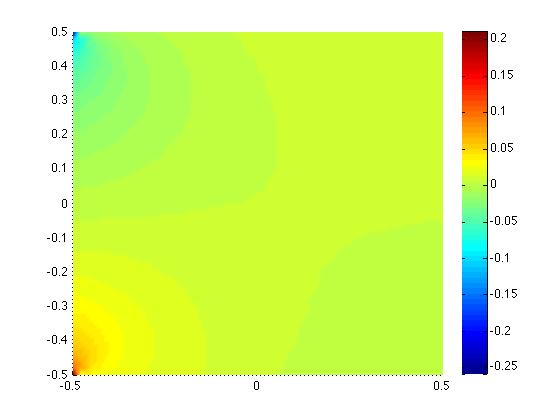}
\caption{Test 3: Computed pressure $p$: surface plot (left) and color plot (right).}\label{test1_2}
\end{figure}


\begin{figure}[th]
\centering
\includegraphics[height=2.4in,width=2.4in]{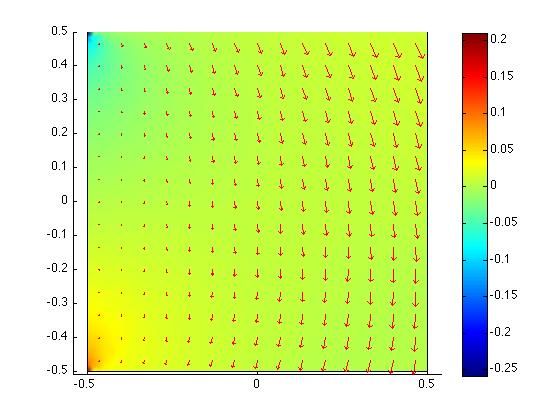}
\includegraphics[height=2.4in,width=2.4in]{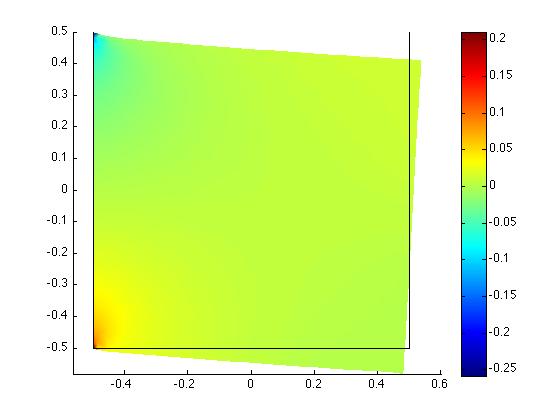}
\caption{Test 3: Arrow plot of the computed displacement (left) and deformation of $\Ome$ (right).}\label{test1_4}
\end{figure}


We remark that the ``locking phenomenon" was observed in the simulation of \cite{pw09} at $T=0.001$
for this problem, namely, the computed pressure exhibits some oscillation at $T=0.001$. 
The reason for the locking phenomenon was explained as follows: when time step $\Delta t$ is small, the 
displacement vector $\mathbf{u}$ is almost divergence free in the short time while the numerical solution 
does not observe this nearly divergence free property, which results in the locking. However, at later times 
the displacement vector is no longer divergence free, so no locking exists at later times.

It is clear that our numerical solution does not exhibits the locking phenomenon at $T=0.001$. 
This is because our multiphysics reformulation weakly imposes the condition $\mbox{div } \mathbf{u}=q$, 
hence, $\mathbf{u}$ automatically becomes nearly divergence free when $q\approx 0$ for $0<t<<1$.
Moreover, the pressure $p$ is not a primary variable anymore in our reformulation, instead, 
$p$ becomes a derivative variable and it is computed using the new primary variables $\xi$ and $\eta$.
Therefore, our numerical methods are insensitive to the regularity of $p$.